\documentclass[a4paper]{article}
\usepackage{geometry}
\usepackage{amssymb}
\usepackage{latexsym}
\usepackage{amsmath}
\usepackage{amsfonts}
\usepackage{amsthm}
\usepackage{indentfirst}
\usepackage{graphicx}
\usepackage[colorlinks=true]{hyperref}
\usepackage{mathrsfs}
\usepackage{chngcntr}
\usepackage{color}

\newtheorem{Thm}{Theorem}[section]
\newtheorem{Cor}[Thm]{Corollary}
\newtheorem{Lem}[Thm]{Lemma}
\newtheorem{Pro}[Thm]{Proposition}

\newcommand{\N}{\mathbb{N}}

\newcommand{\R}{\mathbb{R}}

\numberwithin{equation}{section}

\begin{document}
	
	\title{Existence and concentration of ground state solutions for an exponentially critical Choquard equation involving mixed local-nonlocal operators}
	
	\author{Shaoxiong Chen \quad Min Yang\quad Zhipeng Yang\thanks{Corresponding author: yangzhipeng326@163.com} \\[2pt]
    \small Yunnan Key Laboratory of Modern Analytical Mathematics and Applications, Kunming, China\\
		\small Department of Mathematics, Yunnan Normal University, Kunming, China
	}
	\date{}
	\maketitle
	
	\begin{abstract}
		We study the Choquard equation involving mixed local and nonlocal operators
		\begin{equation*}
			-\varepsilon^{2}\Delta u+\varepsilon^{2s}(-\Delta)^{s}u+V(x)u
		=\varepsilon^{\mu-2}\left(\frac{1}{|x|^{\mu}}*F(u)\right)f(u)
		\quad \text{in }\R^{2},
		\end{equation*}
		where \(\varepsilon>0\), \(s\in(0,1)\), \(0<\mu<2\), \(f\) has Trudinger--Moser critical exponential growth, and \(F(t)=\int_{0}^{t}f(\tau)\,d\tau\).
		By variational methods, combined with the Trudinger--Moser inequality and compactness arguments adapted to the critical growth and the nonlocal interaction term, we prove the existence of ground state solutions and describe their concentration behavior as \(\varepsilon\to0^{+}\).
		
		
	\end{abstract}

%\tableofcontents

\section{Introduction and the main results}

In this paper, we consider the existence and concentration of positive ground state solutions to the mixed local and nonlocal Choquard equation
\begin{equation}\label{eq1.1}
	-\varepsilon^{2} \Delta u+\varepsilon^{2s}(-\Delta)^{s} u+V(x) u
	=\varepsilon^{\mu-2}\left(\frac{1}{|x|^{\mu}} * F(u)\right) f(u)
	\quad \text{in } \R^{2},
\end{equation}
where $\varepsilon>0$, $s\in(0,1)$, $0<\mu<2$, $V:\R^{2}\to\R$ is continuous, $f:\R^{2}\to\R$ is a continuous function, and
\[
F(t)=\int_{0}^{t} f(\tau)\,d\tau.
\]
Here $\Delta$ denotes the Laplacian and $(-\Delta)^{s}$ is the fractional Laplacian defined, up to a positive normalization constant, by
\[
(-\Delta)^{s}u(x)=\operatorname{P.V.}\int_{\R^{2}} \frac{u(x)-u(y)}{|x-y|^{2+2s}}\,dy,
\]
where $\operatorname{P.V.}$ stands for the Cauchy principal value.

The operator in \eqref{eq1.1} combines a second-order local diffusion and a nonlocal diffusion of order $2s$. It is convenient to introduce the unscaled mixed operator
\[
\mathcal{L}=-\Delta+(-\Delta)^{s},
\]
as well as its semiclassical scaling
\[
\mathcal{L}_{\varepsilon}u=-\varepsilon^{2}\Delta u+\varepsilon^{2s}(-\Delta)^{s}u.
\]
The aim of this work is to study semiclassical ground states for a two-dimensional Choquard equation in which the mixed operator interacts with a critical exponential nonlinearity and a nonlocal convolution term.

In recent years, equations involving mixed local and nonlocal operators have received increasing attention. Such models arise in different applied contexts and have stimulated the development of new tools in PDE theory; see, for instance, \cite{2023BiagiJAM,2024BiagiCCM,2024DipierroAMS,2024SuMZ} and the references therein. On bounded domains, Li et al.~\cite{2022LiAM} investigated elliptic problems driven by mixed operators of the form
\[
\left\{\begin{aligned}
	-\Delta u+(-\Delta)^{s} u & =\mu g(x, u)+b(x), && x \in \Omega, \\
	u & \geq 0, && x \in \Omega, \\
	u & =0, && x \in \R^{N} \setminus \Omega,
\end{aligned}\right.
\]
where $\Omega\subset \R^{N}$ is bounded. Using the nonsmooth variational approach developed they obtained existence results under suitable assumptions on $g$ and $b$. Biagi et al.~\cite{2022BiagiPDE} developed a general framework for mixed-order elliptic operators, including existence, maximum principles, and interior and boundary regularity, and further regularity properties were derived in \cite{2022LaMaoAM}. In the whole space, Dipierro et al.~\cite{2025DipierroDCDS} studied the subcritical problem
\[
-\Delta u+(-\Delta)^{s} u+u = u^{r-1} \quad \text{in } \R^{N}, \qquad
u>0 \ \text{in } \R^{N}, \qquad u\in H^{1}(\R^{N}),
\]
with $r\in(1,2^{*})$, where $2^{*}=\frac{2N}{N-2}$ for $N\ge 3$. They proved existence and then characterized qualitative properties such as power-type decay and radial symmetry. Related results also exist for mixed models with nonsingular kernels, motivated in part by applications in animal foraging; see \cite{2019DelPezzoFCAA,2021DipierroPA}.

In parallel, Choquard-type equations have been deeply investigated. These equations originate from Hartree--Fock theory and arise in nonlinear optics and population dynamics, among other areas. In the semiclassical regime, Gao et al.~\cite{2018GaoZAMP} proved the existence and concentration of positive ground states for the fractional Schr\"odinger--Choquard equation
\[
\varepsilon^{2s}(-\Delta)^{s} u+V(x)u=\bigl(I_{\alpha}*|u|^{p}\bigr)|u|^{p-2} u
\quad \text{in } \R^{N},
\]
and Ambrosio \cite{2019AmbrosioPA} studied existence, multiplicity, and concentration phenomena for fractional Choquard equations. For the local Choquard case, Yang and Ding \cite{2013YangCPAA} considered
\[
-\varepsilon^{2} \Delta u+V(x) u=\left(\frac{1}{|x|^{\mu}} * u^{p}\right) u^{p-1}
\quad \text{in } \R^{3},
\]
with $0<\mu<3$ and $\frac{6-\mu}{3}<p<6-\mu$, and obtained solutions for small $\varepsilon$ via the Mountain Pass theorem under appropriate assumptions on $V$.

Choquard equations involving mixed operators have only recently begun to be studied systematically. Anthal \cite{2023AnthalJMAA} investigated a mixed operator Choquard problem on bounded domains with a Hardy--Littlewood--Sobolev critical exponent,
\[
\left\{\begin{array}{l}
	\mathcal{L} u=\left(\int_{\Omega} \frac{|u(y)|^{2_\mu^*}}{|x-y|^\mu}\,dy\right)|u|^{2_\mu^*-2} u+\lambda u^p \quad \text{in } \Omega, \\
	u \equiv 0 \quad \text{in } \mathbb{R}^n \backslash \Omega,\qquad u \geq 0 \quad \text{in } \Omega,
\end{array}\right.
\]
where $\Omega\subset \mathbb{R}^n$ has $C^{1,1}$ boundary, $n\ge 3$, $0<\mu<n$, $p \in\left[1,2^*-1\right)$, $2_\mu^*=\frac{2n-\mu}{n-2}$ and $2^*=\frac{2n}{n-2}$. By variational methods, the author established a mixed Hardy--Littlewood--Sobolev inequality and showed that its best constant coincides with the classical one but is not attained. Using refined energy estimates and the Pohozaev identity, the work provided existence and nonexistence results depending on the range of the parameter $\lambda$.
Kirane \cite{2025KiraneFCAA} investigated the mass decay behavior for a semilinear heat equation driven by a mixed local--nonlocal operator,
\[
\left\{\begin{array}{l}
	\partial_t u+t^\beta \mathcal{L} u=-h(t) u^p,\\
	\mathcal{L}=-\Delta+(-\Delta)^{\alpha / 2},\qquad \alpha \in(0,2),
\end{array}\right.
\]
and identified a critical exponent separating different asymptotic regimes.
Giacomoni \cite{2025GiacomoniNNDE} studied normalized solutions to a Choquard equation involving mixed operators under an $L^{2}$-constraint,
\[
\left\{\begin{array}{l}
	\mathcal{L} u+u =\mu\left(I_\alpha *|u|^p\right)|u|^{p-2} u \quad \text{in } \mathbb{R}^n, \\
	\|u\|_2^2=\tau,
\end{array}\right.
\]
where $\mathcal{L}=-\Delta+\lambda(-\Delta)^s$ with $s \in(0,1)$ and $\lambda>0$, and obtained existence, regularity, and equivalence results between normalized solutions and ground states in suitable parameter ranges. Constantin \cite{2026ConstantinJMAA} studied a doubly degenerate parabolic equation involving the mixed local--nonlocal nonlinear operator
\[
\mathcal{A}_\mu u=-\Delta_p u+\mu(-\Delta)_q^s u,
\]
and established existence, uniqueness and qualitative behavior for weak-mild solutions, including stabilization, extinction and blow-up in finite time under appropriate conditions on the nonlinearities.

More broadly, current research on mixed operators has been focusing on interior regularity and maximum principles (see, for example, \cite{2025BiswasJAM,2022XavierARMA,2024DeMA}), boundary Harnack principles \cite{2012ChenAMS}, boundary regularity and overdetermined problems \cite{2023BiswasAMPA}, qualitative properties of solutions \cite{2021BiagiESA}, existence and asymptotics (see, for example, \cite{2022SalortDIE,2022MugnaiMRCP,2022GarainNA,2023DipierroAIAN,2022DipierroAA}), and shape optimization problems \cite{2023BiagiJAM,2019GoelAMS}.

Motivated by these developments, we investigate in this paper a two-dimensional Choquard equation involving mixed operators and critical exponential growth. The central question is whether ground state solutions to \eqref{eq1.1} exist and concentrate as $\varepsilon\to0$ when both the local and nonlocal diffusions are present. The main difficulties come from the critical Trudinger--Moser regime in dimension two, the nonlocal convolution term, and the lack of compactness produced by translations in $\R^{2}$.

For the purpose of looking for positive solution, we always suppose that $f(t)=0$ for $t\leq 0$. In addition, we assume that the nonlinearity $f$ satisfies:
\begin{itemize}
	\item[($f_1$)] $f(t)=o\bigl(t^{\frac{2-\mu}{2}}\bigr)$ as $t\to0$;
	
	\item[($f_2$)] $f(t)$ has critical exponential growth at $+\infty$ in the Trudinger--Moser sense:
	\begin{equation*}
		\lim_{t\to+\infty}\frac{f(t)}{\mathrm{e}^{\alpha t^{2}}}
		=
		\begin{cases}
			0, & \forall\,\alpha>4\pi,\\[4pt]
			+\infty, & \forall\,\alpha<4\pi;
		\end{cases}
	\end{equation*}
	
	\item[($f_3$)] there exists $\theta>1$ such that
	\[
	f(t)\,t \ge \theta F(t) \ge 0
	\quad \text{for all } t>0;
	\]
	
	\item[($f_4$)] the map $t\mapsto f(t)$ is nondecreasing on $(0,+\infty)$;
	
	\item[($f_5$)]
	\[
	\lim_{t\to+\infty}\frac{t f(t) F(t)}{\mathrm{e}^{8\pi t^{2}}}
	\ge \beta,
	\quad \text{with }
	\beta>\frac{(2-\mu)(3-\mu)(4-\mu)^{2}(1+C_{s})}{16\pi^{2}\rho^{4-\mu}}
	\,\mathrm{e}^{\frac{4-\mu}{4}(a+C_{s})\rho^{2}};
	\]
	
	\item[($f_6$)] there exist constants $M_{0}>0$ and $t_{0}>0$ such that
	\[
	F(t)\le M_{0}|f(t)| \quad \text{for all } t\ge t_{0}.
	\]
\end{itemize}
Here $a$, $C_{s}$ and $\rho$ are positive constants that will be fixed later in the variational construction.

For the potential $V\in C(\R^2)$ we assume that
\begin{itemize}
	\item [$(V)$] \(0<V_{0}=\inf\limits_{x\in\R^{2}} V(x)<V_{\infty}=\liminf\limits_{|x|\to+\infty} V(x)<+\infty.\)
\end{itemize}
This type of condition was first introduced by Rabinowitz \cite{1992RabinowitzAM} and is widely used to recover compactness and to describe concentration near the global minima of $V$.

Under these assumptions, we obtain the following result.

\begin{Thm}\label{thm1.1}
	Assume that $(f_1)$--$(f_6)$ and $(V)$ hold. Then there exists $\varepsilon_{0}>0$ such that, for every $\varepsilon\in(0,\varepsilon_{0})$, problem \eqref{eq1.1} admits at least one positive ground state solution $u_{\varepsilon}$. Moreover, if $\eta_{\varepsilon}\in\R^{2}$ is a global maximum point of $u_{\varepsilon}$, then
	\[
	\lim_{\varepsilon\to0} V(\eta_{\varepsilon})=V_{0}.
	\]
\end{Thm}

The rest of the paper is organized as follows. In Section 2 we introduce the variational framework and collect the main analytical tools. In Section 3 we derive quantitative estimates for the minimax level. Section 4 is devoted to the autonomous problem with constant potential $V_0$, where we prove the existence of a ground state solution. In Section 5 we treat the singularly perturbed problem and establish the existence of ground state solutions for $\varepsilon>0$ small. Finally, in Section~6 we analyze the concentration behavior as $\varepsilon \rightarrow 0$, proving the compactness of translated sequences and locating the concentration points near the set $M=\left\{x \in \mathbb{R}^2: V(x)=V_0\right\}$.
\vspace{10pt}

\textbf{Notation.}
Throughout the paper we use the following notation.
\begin{itemize}
	\item[$\bullet$] $B_{R}(x)$ denotes the open ball with radius $R>0$ centered at $x\in\mathbb{R}^2$.
	\item[$\bullet$] The symbols $C$ and $C_i$ ($i\in\mathbb{N}^+$) denote positive constants whose value may change from line to line.
	\item[$\bullet$] The arrows ``$\rightarrow$'' and ``$\rightharpoonup$'' stand for strong convergence and weak convergence, respectively.
	\item[$\bullet$] $o_n(1)$ denotes a quantity that tends to $0$ as $n\to \infty$.
	\item[$\bullet$] For $r\geq 1$, $\|u\|_r=\big(\int_{\mathbb{R}^2}|u|^r\,dx\big)^{1/r}$ is the norm of $u$ in $L^r(\mathbb{R}^2)$.
	\item[$\bullet$] $\|u\|_\infty=\mathrm{ess}\sup\limits_{x\in\mathbb{R}^2}|u(x)|$ is the norm of $u$ in $L^\infty(\mathbb{R}^2)$.
\end{itemize}

\section{Preliminary results}

Throughout this section we assume that the potential $V$ and the nonlinearity $f$ satisfy assumptions $(V)$ and $(f_1)$--$(f_6)$. The Sobolev space $H^{1}(\R^{2})$ is defined by
\[
H^{1}(\R^{2})
=\bigl\{u\in L^{2}(\R^{2}) : \nabla u\in L^{2}(\R^{2};\R^{2})\bigr\},
\]
where $\nabla u$ denotes the weak gradient of $u$. Equipped with the norm
\[
\|u\|_{H^{1}(\R^{2})}
=\Bigl(\int_{\R^{2}}\bigl(|u|^{2}+|\nabla u|^{2}\bigr)\,dx\Bigr)^{\frac12},
\]
$H^{1}(\R^{2})$ is a Hilbert space.

For $s\in(0,1)$, the fractional Sobolev space $H^{s}(\R^{2})$ is defined by
\[
H^{s}(\R^{2})
=\Bigl\{u\in L^{2}(\R^{2}) :
\int_{\R^{2}}\int_{\R^{2}}\frac{|u(x)-u(y)|^{2}}{|x-y|^{2+2s}}\,dx\,dy<\infty\Bigr\},
\]
endowed with the norm
\[
\|u\|_{H^{s}(\R^{2})}
=\Bigl(\int_{\R^{2}}|u|^{2}\,dx
+\frac{C(n,s)}{2}\int_{\R^{2}}\int_{\R^{2}}\frac{|u(x)-u(y)|^{2}}{|x-y|^{2+2s}}\,dx\,dy\Bigr)^{\frac12},
\]
and the Gagliardo seminorm
\[
[u]_{s}
=\Bigl(\frac{C(n,s)}{2}\int_{\R^{2}}\int_{\R^{2}}\frac{|u(x)-u(y)|^{2}}{|x-y|^{2+2s}}\,dx\,dy\Bigr)^{\frac12}.
\]
Moreover, for $u\in C_c^\infty(\R^{2})$ the fractional Laplacian can be written, with the normalization used in this paper, as
\[
(-\Delta)^{s}u(x)=\operatorname{P.V.}\int_{\R^{2}} \frac{u(x)-u(y)}{|x-y|^{2+2s}}\,dy
=-\frac12\int_{\R^{2}}\frac{u(x+y)+u(x-y)-2u(x)}{|y|^{2+2s}}\,dy,
\]
see for instance~\cite{2012DiNezzaBSM}. In particular, for $u,v\in H^{s}(\R^{2})$ one has the identity
\[
\int_{\R^{2}} (-\Delta)^{s}u\,v\,dx
=\frac12\int_{\R^{2}}\int_{\R^{2}}\frac{(u(x)-u(y))(v(x)-v(y))}{|x-y|^{2+2s}}\,dx\,dy.
\]

The following lemma can be found in~\cite{2023AnthalJMAA}.

\begin{Lem}\label{Lem2.1}
	Let $0<s<1$. Then $H^{1}(\R^{2})$ is continuously embedded into $H^{s}(\R^{2})$, that is, there exists a constant $C_{s}>0$ such that, for every $u\in H^{1}(\R^{2})$,
	\[
	\frac12[u]_{s}^{2}\le C_{s}\,\|u\|_{H^{1}(\R^{2})}^{2}
	=C_{s}\bigl(\|u\|_{L^{2}(\R^{2})}^{2}+\|\nabla u\|_{L^{2}(\R^{2})}^{2}\bigr).
	\]
\end{Lem}

The presence of both local and nonlocal terms in \eqref{eq1.1} naturally leads us to consider the space
\[
W_{\varepsilon}
=\Bigl\{u\in H^{1}(\R^{2}) : \int_{\R^{2}} V(\varepsilon x)\,u^{2}(x)\,dx<\infty\Bigr\},
\]
endowed with the inner product
\[
(u,v)_{\varepsilon}
=\int_{\R^{2}}\nabla u\cdot\nabla v\,dx
+\frac12\int_{\R^{2}}\int_{\R^{2}}\frac{(u(x)-u(y))(v(x)-v(y))}{|x-y|^{2+2s}}\,dx\,dy
+\int_{\R^{2}}V(\varepsilon x)\,u(x)v(x)\,dx,
\]
and the associated norm $\|u\|_{\varepsilon}=(u,u)_{\varepsilon}^{1/2}$, namely
\[
\|u\|_{\varepsilon}^{2}
=\int_{\R^{2}}|\nabla u|^{2}\,dx
+\frac12\int_{\R^{2}}\int_{\R^{2}}\frac{|u(x)-u(y)|^{2}}{|x-y|^{2+2s}}\,dx\,dy
+\int_{\R^{2}}V(\varepsilon x)\,u^{2}(x)\,dx.
\]
By Lemma~\ref{Lem2.1} and assumption $(V)$, the norm $\|\cdot\|_{\varepsilon}$ is equivalent on $W_{\varepsilon}$ to
\[
\|u\|_{\varepsilon,0}^{2}
=\int_{\R^{2}}|\nabla u|^{2}\,dx
+\int_{\R^{2}}V(\varepsilon x)\,u^{2}(x)\,dx.
\]
In particular, since $V(\varepsilon x)\ge V_{0}>0$, one has
\[
\|u\|_{H^{1}(\R^{2})}\le C\,\|u\|_{\varepsilon}
\qquad \text{for all }u\in W_{\varepsilon}.
\]

Making the change of variables $x\mapsto \varepsilon x$ in \eqref{eq1.1}, we obtain the equivalent problem
\begin{equation}\label{eq2.1}
	-\Delta u+(-\Delta)^{s} u+V(\varepsilon x) u
	=\left(\frac{1}{|x|^{\mu}}*F(u)\right) f(u)
	\quad \text{in } \R^{2}.
\end{equation}
If $u$ is a solution of \eqref{eq2.1}, then $v(x)=u(x/\varepsilon)$ is a solution of \eqref{eq1.1}.

Problem \eqref{eq2.1} has a variational structure: its weak solutions correspond to the critical points of the functional
\[
\mathcal{J}_{\varepsilon}(u)
=\frac{1}{2}\|u\|_{\varepsilon}^{2}
-\frac{1}{2}\int_{\R^{2}}\left(\frac{1}{|x|^{\mu}}*F(u)\right)F(u)\,dx,
\]
where $F(t)=\int_{0}^{t}f(\tau)\,d\tau$. Moreover, $\mathcal{J}_{\varepsilon}\in C^{1}(W_{\varepsilon},\R)$. We define the associated Nehari manifold by
\[
\mathcal{N}_{\varepsilon}
=\bigl\{u\in W_{\varepsilon}\setminus\{0\} : \mathcal{G}(u)=0\bigr\},
\]
where
\[
\mathcal{G}(u)
=\langle \mathcal{J}_{\varepsilon}'(u),u\rangle
=\|u\|_{\varepsilon}^{2}
-\int_{\R^{2}}\left(\frac{1}{|x|^{\mu}}*F(u)\right)f(u)\,u\,dx.
\]

The first version of the Trudinger--Moser inequality in $\R^{2}$ was established by Cao, see~\cite{1992CaoPDF}; see also~\cite{2000AdachiAMS,2014CassaniJFA,1996MarcosDIE} and the references therein. It can be stated as follows.

\begin{Pro}\label{pro2.2}
	If $\alpha>0$ and $u\in H^{1}(\R^{2})$, then
	\[
	\int_{\R^{2}}\bigl(\mathrm{e}^{\alpha u^{2}}-1\bigr)\,dx<\infty.
	\]
	Moreover, if $\alpha<4\pi$ and $\|u\|_{2}\le M<\infty$, then there exists a constant $C_{1}=C_{1}(M,\alpha)>0$ such that
	\[
	\sup_{\|\nabla u\|_{2}\le1,\;\|u\|_{2}\le M}
	\int_{\R^{2}}\bigl(\mathrm{e}^{\alpha u^{2}}-1\bigr)\,dx
	\le C_{1}.
	\]
\end{Pro}

\begin{Lem}\cite{2001LiebAMS}\label{Lem2.3}
	Let $t,r>1$ and $0<\mu<N$ be such that
	\[
	\frac{1}{t}+\frac{\mu}{N}+\frac{1}{r}=2.
	\]
	If $\varphi\in L^{t}(\R^{N})$ and $\psi\in L^{r}(\R^{N})$, then there exists a constant $C(t,N,\mu,r)>0$, independent of $\varphi$ and $\psi$, such that
	\[
	\int_{\R^{N}}\left(\frac{1}{|x|^{\mu}}*\varphi\right)(x)\psi(x)\,dx
	\le C(t,N,\mu,r)\,\|\varphi\|_{t}\,\|\psi\|_{r}.
	\]
	In particular, when $N=2$ and $t=r=\frac{4}{4-\mu}$, one has
	\[
	\int_{\R^{2}}\left(\frac{1}{|x|^{\mu}}*F(u)\right)F(u)\,dx
	\le C_{\mu}\,\|F(u)\|_{\frac{4}{4-\mu}}^{2},
	\]
	where $C_{\mu}>0$ depends only on $\mu$.
\end{Lem}

\begin{Lem}\cite{1997MattnerTAMS}\label{Lem2.4}
	For $\varphi,\psi\in L_{\mathrm{loc}}^{1}(\R^{2})$ such that the integrals below are finite, one has
	\[
	\int_{\R^{2}}\left(\frac{1}{|x|^{\mu}}*\varphi\right)(x)\psi(x)\,dx
	\le
	\left(\int_{\R^{2}}\left(\frac{1}{|x|^{\mu}}*\varphi\right)(x)\varphi(x)\,dx\right)^{\frac12}
	\left(\int_{\R^{2}}\left(\frac{1}{|x|^{\mu}}*\psi\right)(x)\psi(x)\,dx\right)^{\frac12}.
	\]
\end{Lem}

\begin{Lem}\label{Lem2.5}
	Let $u\in W_{\varepsilon}$, $k>0$, $q>0$, and assume that
	\[
	\|u\|_{H^{1}(\R^{2})}\le M
	\quad\text{and}\quad
	\frac{4}{4-\mu}\,k\,M^{2}<4\pi.
	\]
	Then there exists a constant $C=C(k,M,q)>0$ such that
	\[
	\int_{\R^{2}}\bigl(\bigl(\mathrm{e}^{k u^{2}}-1\bigr)|u|^{q}\bigr)^{\frac{4}{4-\mu}}\,dx
	\le C\,\|u\|_{\varepsilon}^{\frac{4q}{4-\mu}}.
	\]
\end{Lem}

\begin{proof}
	Let $p=\frac{4}{4-\mu}$. Choose $r>1$ and set $r'=\frac{r}{r-1}$ so that
	\[
	\frac{p q r'}{1}\ge 2
	\quad\text{and}\quad
	r\,p\,k\,M^{2}<4\pi.
	\]
	Using $(\mathrm{e}^{\tau}-1)^{p}\le \mathrm{e}^{p\tau}-1$ and $(\mathrm{e}^{\tau}-1)^{r}\le \mathrm{e}^{r\tau}-1$ for $\tau\ge0$, H\"older's inequality and Proposition~\ref{pro2.2}, we obtain
	\[
	\begin{aligned}
		\int_{\R^{2}}\bigl(\bigl(\mathrm{e}^{k u^{2}}-1\bigr)|u|^{q}\bigr)^{p}\,dx
		&\le \int_{\R^{2}}\bigl(\mathrm{e}^{p k u^{2}}-1\bigr)\,|u|^{p q}\,dx\\
		&\le \left(\int_{\R^{2}}\bigl(\mathrm{e}^{p k u^{2}}-1\bigr)^{r}\,dx\right)^{\frac{1}{r}}
		\left(\int_{\R^{2}}|u|^{p q r'}\,dx\right)^{\frac{1}{r'}}.
	\end{aligned}
	\]
	Write $u = A v$ with $A=\|u\|_{H^{1}(\R^{2})}\le M$ and $\|v\|_{H^{1}(\R^{2})}=1$. Then
	\[
	\bigl(\mathrm{e}^{p k u^{2}}-1\bigr)^{r}
	\le \mathrm{e}^{r p k A^{2} v^{2}}-1,
	\]
	and by the choice of $r$ the parameter $r p k M^{2}$ is strictly less than $4\pi$. Hence Proposition~\ref{pro2.2} yields
	\[
	\int_{\R^{2}}\bigl(\mathrm{e}^{r p k A^{2} v^{2}}-1\bigr)\,dx
	\le\sup_{\|\nabla v\|_{2}\le1,\;\|v\|_{2}\le 1}
	\int_{\R^{2}}\bigl(\mathrm{e}^{r p k M^{2} v^{2}}-1\bigr)\,dx
	\leq  C(k,M,r)
	\]
	for some constant $C(k,M,r)>0$ independent of $u$.
	
	On the other hand, since $p q r'\ge2$ and $u\in H^{1}(\R^{2})$, the continuous embedding $H^{1}(\R^{2})\hookrightarrow L^{m}(\R^{2})$ for all $m\ge2$ gives
	\[
	\left(\int_{\R^{2}}|u|^{p q r'}\,dx\right)^{\frac{1}{r'}}
	\le C\,\|u\|_{H^{1}(\R^{2})}^{p q}.
	\]
	Combining the last two estimates and using $\|u\|_{H^{1}(\R^{2})}\le C\,\|u\|_{\varepsilon}$ on $W_{\varepsilon}$, we conclude that
	\[
	\int_{\R^{2}}\bigl(\bigl(\mathrm{e}^{k u^{2}}-1\bigr)|u|^{q}\bigr)^{p}\,dx
	\le C(k,M,q)\,\|u\|_{\varepsilon}^{p q}.
	\]
	This completes the proof.
\end{proof}

\begin{Lem}\label{Lem2.6}
	For any $\varepsilon>0$, the functional $\mathcal{J}_{\varepsilon}$ satisfies:
	\begin{itemize}
		\item[(i)] There exist $\rho>0$ and $\alpha_{1}>0$ such that $\mathcal{J}_{\varepsilon}(u)\ge \alpha_{1}$ for all $u\in W_{\varepsilon}$ with $\|u\|_{\varepsilon}=\rho$.
		\item[(ii)] There exists $e\in W_{\varepsilon}$ with $\|e\|_{\varepsilon}>\rho$ such that $\mathcal{J}_{\varepsilon}(e)<0$.
	\end{itemize}
\end{Lem}

\begin{proof}
	(i) By $(f_{1})$--$(f_{2})$, there exist $q>1$ and $k>0$ such that for every $\eta>0$ there is $C_{\eta}>0$ with
	\begin{equation}\label{eq2.2}
		|F(t)|
		\le \eta\,|t|^{\frac{4-\mu}{2}}
		+C_{\eta}\,|t|^{q}\bigl(\mathrm{e}^{k t^{2}}-1\bigr)
		\quad\text{for all }t\in\R.
	\end{equation}
	Using Lemma~\ref{Lem2.3} with $N=2$ and $t=r=\frac{4}{4-\mu}$, we obtain
	\[
	\int_{\R^{2}}\left(\frac{1}{|x|^{\mu}}*F(u)\right)F(u)\,dx
	\le C_{\mu}\,\|F(u)\|_{\frac{4}{4-\mu}}^{2}.
	\]
	Fix $\eta>0$. By \eqref{eq2.2} and $(a+b)^{p}\le 2^{p-1}(a^{p}+b^{p})$ for $a,b\ge0$, $p\ge1$, we deduce
	\[
	\begin{aligned}
		\|F(u)\|_{\frac{4}{4-\mu}}^{2}
		&\le C\left(\int_{\R^{2}}\Bigl(\eta\,|u|^{\frac{4-\mu}{2}}
		+C_{\eta}\,|u|^{q}\bigl(\mathrm{e}^{k u^{2}}-1\bigr)\Bigr)^{\frac{4}{4-\mu}}\,dx\right)^{\frac{4-\mu}{2}}\\
		&\le C_{1}\,\|u\|_{2}^{4-\mu}
		+C_{2}\left(\int_{\R^{2}}\bigl(|u|^{q}\bigl(\mathrm{e}^{k u^{2}}-1\bigr)\bigr)^{\frac{4}{4-\mu}}\,dx\right)^{\frac{4-\mu}{2}}.
	\end{aligned}
	\]
	Let $C_H>0$ be such that $\|u\|_{H^{1}(\R^{2})}\le C_H\|u\|_{\varepsilon}$ for all $u\in W_{\varepsilon}$. Choose $\rho>0$ so small that, whenever $\|u\|_{\varepsilon}=\rho$, one has
	\[
	\|u\|_{H^{1}(\R^{2})}\le M
	\quad\text{with}\quad
	M=C_H\rho
	\quad\text{and}\quad
	\frac{4}{4-\mu}\,k\,M^{2}<4\pi,
	\]
	so that Lemma~\ref{Lem2.5} applies. Using $\|u\|_{2}\le V_{0}^{-\frac12}\|u\|_{\varepsilon}$ and Lemma~\ref{Lem2.5}, for $\|u\|_{\varepsilon}=\rho$ we obtain
	\[
	\int_{\R^{2}}\left(\frac{1}{|x|^{\mu}}*F(u)\right)F(u)\,dx
	\le C_{3}\,\|u\|_{\varepsilon}^{4-\mu}
	+C_{4}\,\|u\|_{\varepsilon}^{2q}.
	\]
	Therefore, for $\|u\|_{\varepsilon}=\rho$,
	\[
	\mathcal{J}_{\varepsilon}(u)
	\ge \frac{1}{2}\rho^{2}
	-\frac{C_{3}}{2}\rho^{4-\mu}
	-\frac{C_{4}}{2}\rho^{2q}.
	\]
	Since $4-\mu>2$ and $2q>2$, choosing $\rho$ smaller if necessary we get
	\[
	\mathcal{J}_{\varepsilon}(u)\ge \alpha_{1}>0,
	\]
	which yields (i).
	
	(ii) Let $u_{0}\in W_{\varepsilon}$ satisfy $u_{0}\ge0$ and $u_{0}\not\equiv0$. Set
	\[
	\Psi(u)=\int_{\R^{2}}\left(\frac{1}{|x|^{\mu}}*F(u)\right)F(u)\,dx.
	\]
	For $t>0$ define
	\[
	A(t)=\Psi\left(\frac{t u_{0}}{\|u_{0}\|_{\varepsilon}}\right).
	\]
	Then $A(t)\ge0$ for $t>0$. Moreover, using the symmetry of the convolution form one computes
	\[
	A'(t)
	=\frac{2}{\|u_{0}\|_{\varepsilon}}\int_{\R^{2}}
	\left(\frac{1}{|x|^{\mu}}*F\left(\frac{t u_{0}}{\|u_{0}\|_{\varepsilon}}\right)\right)
	f\left(\frac{t u_{0}}{\|u_{0}\|_{\varepsilon}}\right)u_{0}\,dx.
	\]
	Rewriting,
	\[
	\begin{aligned}
		A'(t)
		&=\frac{2}{t}\int_{\R^{2}}
		\left(\frac{1}{|x|^{\mu}}*F\left(\frac{t u_{0}}{\|u_{0}\|_{\varepsilon}}\right)\right)
		f\left(\frac{t u_{0}}{\|u_{0}\|_{\varepsilon}}\right)
		\frac{t u_{0}}{\|u_{0}\|_{\varepsilon}}\,dx\\
		&\ge \frac{2\theta}{t}\int_{\R^{2}}
		\left(\frac{1}{|x|^{\mu}}*F\left(\frac{t u_{0}}{\|u_{0}\|_{\varepsilon}}\right)\right)
		F\left(\frac{t u_{0}}{\|u_{0}\|_{\varepsilon}}\right)\,dx
		=\frac{2\theta}{t}\,A(t),
	\end{aligned}
	\]
	where we used $(f_{3})$. Thus for $t>0$,
	\[
	\frac{A'(t)}{A(t)}\ge \frac{2\theta}{t}
	\quad\text{whenever }A(t)>0.
	\]
	Integrating from $1$ to $\sigma>1$ gives
	\[
	A(\sigma)\ge A(1)\,\sigma^{2\theta}\quad\text{for all }\sigma\ge1.
	\]
	Taking $\sigma=t\|u_{0}\|_{\varepsilon}$ with $t\ge \frac{1}{\|u_{0}\|_{\varepsilon}}$, we obtain
	\[
	\Psi(tu_{0})
	=A\bigl(t\|u_{0}\|_{\varepsilon}\bigr)
	\ge \Psi\Bigl(\frac{u_{0}}{\|u_{0}\|_{\varepsilon}}\Bigr)\,\|u_{0}\|_{\varepsilon}^{2\theta}\,t^{2\theta}.
	\]
	Therefore,
	\[
	\mathcal{J}_{\varepsilon}(t u_{0})
	\le \frac{t^{2}}{2}\,\|u_{0}\|_{\varepsilon}^{2}
	-\frac{1}{2}\Psi\Bigl(\frac{u_{0}}{\|u_{0}\|_{\varepsilon}}\Bigr)\,\|u_{0}\|_{\varepsilon}^{2\theta}\,t^{2\theta}.
	\]
	Since $\theta>1$, the right-hand side tends to $-\infty$ as $t\to+\infty$. Hence we can choose $t_{0}>0$ large enough such that, setting $e=t_{0}u_{0}$, we have $\|e\|_{\varepsilon}>\rho$ and $\mathcal{J}_{\varepsilon}(e)<0$. This proves (ii).
\end{proof}

Combining Lemma~\ref{Lem2.6} with the mountain pass theorem, we obtain a $(PS)$ sequence $\{u_{n}\}\subset W_{\varepsilon}$ such that
\[
\mathcal{J}_{\varepsilon}(u_{n})\to c_{\varepsilon}
\quad\text{and}\quad
\mathcal{J}_{\varepsilon}'(u_{n})\to0 \ \text{ in } W_{\varepsilon}^{*},
\]
where the minimax level is given by
\[
c_{\varepsilon}
=\inf_{g\in\Gamma}\,\sup_{t\in[0,1]}\mathcal{J}_{\varepsilon}(g(t))>0,
\]
and
\[
\Gamma
=\bigl\{g\in C([0,1],W_{\varepsilon}) : g(0)=0,\ \mathcal{J}_{\varepsilon}(g(1))<0\bigr\}.
\]

\begin{Lem}\label{Lem2.7}
Assume that $f(t)=0$ for all $t\le 0$.
For every $u\in W_{\varepsilon}\setminus\{0\}$ with $u^{+}\not\equiv 0$ there exists a unique $t(u)>0$ such that $t(u)u\in\mathcal{N}_{\varepsilon}$. Moreover,
\[
\mathcal{J}_{\varepsilon}(t(u)u)
=\max_{t\ge0}\mathcal{J}_{\varepsilon}(tu).
\]
\end{Lem}

\begin{proof}
Fix $u\in W_{\varepsilon}\setminus\{0\}$ with $u^{+}\not\equiv 0$ and define $h:[0,\infty)\to\mathbb{R}$ by
\[
h(t)=\mathcal{J}_{\varepsilon}(tu),\qquad t\ge0.
\]
By $(f_{1})$ and Lemma~\ref{Lem2.3}, one has $h(t)>0$ for all $t>0$ sufficiently small.
By $(f_{3})$ and the argument in Lemma~\ref{Lem2.6}(ii), one has $h(t)\to-\infty$ as $t\to+\infty$.
Hence $h$ attains a global maximum at some $t(u)>0$, and at such a point $h'(t(u))=0$.
Since
\[
h'(t)=\bigl\langle \mathcal{J}_{\varepsilon}'(tu),u\bigr\rangle,
\]
we obtain
\[
\bigl\langle \mathcal{J}_{\varepsilon}'(t(u)u),u\bigr\rangle=0.
\]
Because $t(u)>0$,
\[
\bigl\langle \mathcal{J}_{\varepsilon}'(t(u)u),t(u)u\bigr\rangle
=t(u)\,\bigl\langle \mathcal{J}_{\varepsilon}'(t(u)u),u\bigr\rangle=0,
\]
that is, $t(u)u\in\mathcal{N}_{\varepsilon}$. The maximality of $t(u)$ gives
\[
\mathcal{J}_{\varepsilon}(t(u)u)=\max_{t\ge0}\mathcal{J}_{\varepsilon}(tu).
\]

Now we prove the uniqueness. Since $f(t)=0$ for $t\le 0$, we have $F(t)=0$ for $t\le 0$, and thus
\[
F(tu)=F(tu^{+}),\qquad f(tu)=f(tu^{+})
\quad\text{a.e. in }\mathbb{R}^{2},\ \forall t\ge 0.
\]
Writing $h'(t)=0$ in the symmetric double-integral form, we have that $h'(t)=0$ is equivalent to
\begin{equation}\label{eq2.3}
\|u\|_{\varepsilon}^{2}
=
\iint_{\mathbb{R}^{2}\times\mathbb{R}^{2}}
\left(\frac{F(tu^{+}(y))}{t\,u^{+}(y)}\right)\,
f(tu^{+}(x))\,
\frac{u^{+}(x)\,u^{+}(y)}{|x-y|^{\mu}}\,dx\,dy,
\end{equation}
where we set $\frac{F(tu^{+}(y))}{t\,u^{+}(y)}=0$ whenever $u^{+}(y)=0$.
Denote the right-hand side of \eqref{eq2.3} by $R(t)$.

Using $(f_{3})$, there exists $\theta>1$ such that $t f(t)\ge \theta F(t)\ge 0$ for all $t>0$. Hence,
for every $a>0$ the function
\[
t\mapsto \frac{F(ta)}{ta}
\]
is nondecreasing on $(0,\infty)$, and it is strictly increasing on any interval where $F(ta)>0$.
Moreover, by $(f_{4})$, for every $b\ge 0$ the function $t\mapsto f(tb)$ is nondecreasing on $(0,\infty)$.
Therefore, for a.e. $(x,y)$ the integrand in \eqref{eq2.3} is nondecreasing in $t$, and consequently $R(t)$ is nondecreasing on $(0,\infty)$.

Assume by contradiction that there exist $0<t_{1}<t_{2}$ such that $h'(t_{1})=h'(t_{2})=0$.
Then $R(t_{1})=R(t_{2})=\|u\|_{\varepsilon}^{2}>0$.
In particular, the set
\[
E:=\Bigl\{(x,y)\in\mathbb{R}^{2}\times\mathbb{R}^{2}:\ 
u^{+}(x)u^{+}(y)>0,\ f(t_{1}u^{+}(x))>0,\ F(t_{1}u^{+}(y))>0\Bigr\}
\]
has positive measure, otherwise the integrand in \eqref{eq2.3} would vanish a.e. and $R(t_{1})=0$, a contradiction.
For every $(x,y)\in E$, we have $u^{+}(y)>0$ and $F(t_{1}u^{+}(y))>0$, hence
\[
\frac{F(t_{2}u^{+}(y))}{t_{2}u^{+}(y)}>\frac{F(t_{1}u^{+}(y))}{t_{1}u^{+}(y)}.
\]
Also $f(tu^{+}(x))$ is nondecreasing and $f(t_{1}u^{+}(x))>0$ on $E$, hence
\[
f(t_{2}u^{+}(x))\ge f(t_{1}u^{+}(x))>0.
\]
It follows that the integrand in \eqref{eq2.3} is strictly larger at $t_{2}$ than at $t_{1}$ on $E$.
Integrating over $\mathbb{R}^{2}\times\mathbb{R}^{2}$ yields $R(t_{2})>R(t_{1})$, contradicting $R(t_{2})=R(t_{1})$.
Therefore the equation $h'(t)=0$ admits at most one solution $t>0$, and the corresponding $t(u)$ is unique.
\end{proof}

Next, we define the numbers
\[
c_{\varepsilon}^{*}
=\inf_{u\in \mathcal{N}_{\varepsilon}}\mathcal{J}_{\varepsilon}(u),
\qquad
c_{\varepsilon}^{**}
=\inf_{u\in W_{\varepsilon}\setminus\{0\}}\ \max_{t\ge0}\mathcal{J}_{\varepsilon}(tu).
\]

\begin{Lem}\label{Lem2.8}
	For any fixed $\varepsilon>0$ one has
	\[
	c_{\varepsilon}=c_{\varepsilon}^{*}=c_{\varepsilon}^{**}.
	\]
\end{Lem}

\begin{proof}
	First, Lemma~\ref{Lem2.7} implies that for each $u\in W_{\varepsilon}\setminus\{0\}$ one has $t(u)u\in\mathcal{N}_{\varepsilon}$ and
	\[
	\max_{t\ge0}\mathcal{J}_{\varepsilon}(tu)=\mathcal{J}_{\varepsilon}(t(u)u).
	\]
	Hence
	\[
	c_{\varepsilon}^{**}
	=\inf_{u\in W_{\varepsilon}\setminus\{0\}} \mathcal{J}_{\varepsilon}(t(u)u)
	\ge \inf_{w\in\mathcal{N}_{\varepsilon}} \mathcal{J}_{\varepsilon}(w)
	=c_{\varepsilon}^{*}.
	\]
	Conversely, for every $w\in\mathcal{N}_{\varepsilon}$, using Lemma~\ref{Lem2.7} again,we have $\max_{t\ge0}\mathcal{J}_{\varepsilon}(tu)= \mathcal{J}_{\varepsilon}(w)$, hence $$c_{\varepsilon}^{**}\le \max_{t\ge0}\mathcal{J}_{\varepsilon}(tu)= \mathcal{J}_{\varepsilon}(w) \leq c_{\varepsilon}^{*}$$ Therefore $c_{\varepsilon}^{**}=c_{\varepsilon}^{*}$.
	
	To compare with $c_{\varepsilon}$, let $g\in\Gamma$. Since $\mathcal{G}(g(0))=0$ and $\mathcal{J}_{\varepsilon}(g(1))<0$, one has
	\[
	\Psi(g(1))>\|g(1)\|_{\varepsilon}^{2}.
	\]
	Using $(f_{3})$ we obtain
	\[
	\int_{\R^{2}}\left(\frac{1}{|x|^{\mu}}*F(g(1))\right)f(g(1))g(1)\,dx
	\ge \theta\,\Psi(g(1)),
	\]
	hence
	\[
	\mathcal{G}(g(1))
	\le \|g(1)\|_{\varepsilon}^{2}-\theta\,\Psi(g(1))
	< (1-\theta)\|g(1)\|_{\varepsilon}^{2}<0.
	\]
	Similar to lemma~\ref{Lem2.6}$(i)$, there exists $g(\sigma)>0$ sufficiently small, such that $\mathcal{G}(g(\sigma))>0$ , By continuity of $\mathcal{G}\circ g$, there exists $t_{0}\in(0,1)$ such that $g(t_{0})\in\mathcal{N}_{\varepsilon}$. Then
	\[
	\sup_{t\in[0,1]}\mathcal{J}_{\varepsilon}(g(t))\ge \mathcal{J}_{\varepsilon}(g(t_{0}))\ge c_{\varepsilon}^{*}.
	\]
	Taking the infimum over $g\in\Gamma$ yields $c_{\varepsilon}\ge c_{\varepsilon}^{*}$.
	
	On the other hand, fix $u\in W_{\varepsilon}\setminus\{0\}$ and let $t(u)>0$ be given by Lemma~\ref{Lem2.7}. Since $\mathcal{J}_{\varepsilon}(tu)\to-\infty$ as $t\to+\infty$, we can choose $T(u)>t(u)$ such that $\mathcal{J}_{\varepsilon}(T(u)u)<0$. Define $g_u(t)=t\,T(u)u$. Then $g_u\in\Gamma$ and
	\[
	\sup_{t\in[0,1]}\mathcal{J}_{\varepsilon}(g_u(t))
	=\max_{s\in[0,T(u)]}\mathcal{J}_{\varepsilon}(su)
	=\mathcal{J}_{\varepsilon}(t(u)u)
	=\max_{s\ge0}\mathcal{J}_{\varepsilon}(su).
	\]
	Taking the infimum over $u\ne0$ gives
	\[
	c_{\varepsilon}\le \inf_{u\in W_{\varepsilon}\setminus\{0\}}\max_{s\ge0}\mathcal{J}_{\varepsilon}(su)
	=c_{\varepsilon}^{**}.
	\]
	Therefore $c_{\varepsilon}=c_{\varepsilon}^{*}=c_{\varepsilon}^{**}$.
\end{proof}

\section{Estimates for the minimax level}

In this section we introduce an autonomous limit problem and its variational structure, which will be used to compare the minimax level $c_{\varepsilon}$ with a reference level in the semiclassical regime.

Let $a>0$ be a constant. We consider the autonomous Choquard problem
\begin{equation}\label{eq3.1}
	\begin{cases}
		-\Delta u+(-\Delta)^{s}u+a\,u
		=\left(\dfrac{1}{|x|^{\mu}}*F(u)\right)f(u)
		& \text{in } \R^{2},\\[4pt]
		u\in H^{1}(\R^{2}),\quad u>0
		& \text{in } \R^{2}.
	\end{cases}
\end{equation}
Since $H^{1}(\R^{2})$ is continuously embedded into $H^{s}(\R^{2})$, the Gagliardo term is finite for every $u\in H^{1}(\R^{2})$ and the natural energy space is $W_{a}=H^{1}(\R^{2})$ endowed with the norm
\[
\|u\|_{a}^{2}
=\int_{\R^{2}}|\nabla u|^{2}\,dx
+\frac12\int_{\R^{2}}\int_{\R^{2}}
\frac{|u(x)-u(y)|^{2}}{|x-y|^{2+2s}}\,dx\,dy
+\int_{\R^{2}}a\,u^{2}\,dx.
\]

The variational functional associated with \eqref{eq3.1} is
\[
\mathcal{I}_{a}(u)
=\frac{1}{2}\|u\|_{a}^{2}
-\frac{1}{2}\int_{\R^{2}}\left(\frac{1}{|x|^{\mu}}*F(u)\right)F(u)\,dx,
\]
where $F(t)=\int_{0}^{t}f(\tau)\,d\tau$. Then $\mathcal{I}_{a}\in C^{1}(W_{a},\R)$ and its derivative satisfies
\[
\bigl\langle \mathcal{I}_{a}'(u),u\bigr\rangle
=\|u\|_{a}^{2}
-\int_{\R^{2}}\left(\frac{1}{|x|^{\mu}}*F(u)\right)f(u)\,u\,dx.
\]

We define the Nehari manifold associated with $\mathcal{I}_{a}$ by
\[
\mathcal{N}_{a}
=\bigl\{u\in W_{a}\setminus\{0\} :
\bigl\langle \mathcal{I}_{a}'(u),u\bigr\rangle=0\bigr\},
\]
and the corresponding level by
\begin{equation}\label{eq3.2}
	c_{a}
	=\inf_{u\in\mathcal{N}_{a}}\mathcal{I}_{a}(u).
\end{equation}
The basic properties of $c_{a}$ and $\mathcal{N}_{a}$ are analogous to those of $c_{\varepsilon}$ and $\mathcal{N}_{\varepsilon}$.

\begin{Lem}\label{Lem3.1}
	Assume that $(V)$ and $(f_1)$--$(f_6)$ hold. Then the level $c_{a}$ satisfies
	\begin{equation}\label{eq3.3}
		c_{a}<\frac{4-\mu}{8}\bigl(1+C_{s}\bigr).
	\end{equation}
\end{Lem}

\begin{proof}
	Let $\rho>0$ be the constant appearing in $(f_5)$. We introduce the following Moser-type functions $\bar w_{n}$ supported in $B_{\rho}(0)$ (see \cite{2016AlvesJDF}):
	\[
	\bar w_{n}(x)
	=\frac{1}{\sqrt{2\pi}}
	\begin{cases}
		\sqrt{\log n}, & 0\le |x|\le \dfrac{\rho}{n},\\[3pt]
		\dfrac{\log(\rho/|x|)}{\sqrt{\log n}}, & \dfrac{\rho}{n}\le |x|\le \rho,\\[3pt]
		0, & |x|\ge \rho.
	\end{cases}
	\]
	A direct computation gives
	\[
	\int_{\R^{2}}|\nabla\bar w_{n}|^{2}\,dx
	=\int_{\rho/n}^{\rho}\frac{1}{r\log n}\,dr
	=1
	\]
	and, using polar coordinates,
	\[
	\begin{aligned}
		\int_{\R^{2}}|\bar w_{n}|^{2}\,dx
		&=\int_{0}^{\rho/n} r\log n\,dr
		+\int_{\rho/n}^{\rho}\frac{r\,\log^{2}(\rho/r)}{\log n}\,dr\\
		&=\rho^{2}\left(\frac{1}{4\log n}
		-\frac{1}{4n^{2}\log n}
		-\frac{1}{2n^{2}}\right).
	\end{aligned}
	\]
	We set
	\[
	\delta_{n}
	=\rho^{2}\left(\frac{1}{4\log n}
	-\frac{1}{4n^{2}\log n}
	-\frac{1}{2n^{2}}\right).
	\]
	
	By Lemma~\ref{Lem2.1} and the definition of $\|\cdot\|_{a}$, we get
	\[
	\begin{aligned}
		\|\bar w_{n}\|_{a}^{2}
		&=\int_{\R^{2}}|\nabla \bar w_{n}|^{2}\,dx
		+\int_{\R^{2}}a\,\bar w_{n}^{2}\,dx
		+\frac12\int_{\R^{2}}\int_{\R^{2}}
		\frac{|\bar w_{n}(x)-\bar w_{n}(y)|^{2}}{|x-y|^{2+2s}}\,dx\,dy\\
		&\le 1+a\,\delta_{n}
		+C_{s}\bigl(1+\delta_{n}\bigr)\\
		&=1+C_{s}+\bigl(a+C_{s}\bigr)\delta_{n}.
	\end{aligned}
	\]
	Define
	\[
	w_{n}(x)
	=\frac{\bar w_{n}(x)}{\sqrt{1+C_{s}+(a+C_{s})\delta_{n}}}.
	\]
	Then
	\begin{equation}\label{eq3.4}
		\|w_{n}\|_{a}^{2}\le 1.
	\end{equation}
	
	To prove \eqref{eq3.3}, it is enough to show that there exists $n$ such that
	\begin{equation}\label{eq3.5}
		\max_{t\ge0}\mathcal{I}_{a}(t w_{n})
		<\frac{4-\mu}{8}\bigl(1+C_{s}\bigr).
	\end{equation}
	Arguing by contradiction, assume that \eqref{eq3.5} fails. Then, for every $n$, there exists $t_{n}>0$ such that
	\begin{equation}\label{eq3.6}
		\max_{t\ge0}\mathcal{I}_{a}(t w_{n})
		=\mathcal{I}_{a}(t_{n}w_{n})
		\ge\frac{4-\mu}{8}\bigl(1+C_{s}\bigr),
	\end{equation}
	and $t_{n}$ satisfies
	\[
	\left.\frac{d}{dt}\mathcal{I}_{a}(t w_{n})\right|_{t=t_{n}}=0.
	\]
	Computing the derivative, we obtain
	\begin{equation}\label{eq3.7}
		t_{n}^{2}\|w_{n}\|_{a}^{2}
		=\int_{\R^{2}}\left(\frac{1}{|x|^{\mu}}*F(t_{n}w_{n})\right)
		f(t_{n}w_{n})\,t_{n}w_{n}\,dx.
	\end{equation}
	
	From \eqref{eq3.6} and the fact that the Choquard term is nonnegative, we have
	\[
	\frac{1}{2}t_{n}^{2}\|w_{n}\|_{a}^{2}
	\ge \frac{4-\mu}{8}\bigl(1+C_{s}\bigr),
	\]
	so by \eqref{eq3.4},
	\begin{equation}\label{eq3.8}
		t_{n}^{2}
		\ge \frac{4-\mu}{4}\bigl(1+C_{s}\bigr).
	\end{equation}
	
	Next we use $(f_{5})$. By the definition of $\beta$ in $(f_{5})$, for every $\varepsilon>0$ there exists $t_{\varepsilon}>0$ such that
	\begin{equation}\label{eq3.9}
		tf(t)\,F(t)\ge (\beta-\varepsilon)\,\mathrm{e}^{8\pi t^{2}}
		\quad\text{for all }t\ge t_{\varepsilon}.
	\end{equation}
	On $B_{\rho/n}$ the function $w_{n}$ is constant and equal to
	\[
	w_{n}
	=\frac{1}{\sqrt{2\pi}}
	\frac{\sqrt{\log n}}{\sqrt{1+C_{s}+(a+C_{s})\delta_{n}}}.
	\]
	Combining this with \eqref{eq3.8} and $\delta_n\to0$, we have $t_{n}w_{n}\to+\infty$ on $B_{\rho/n}$ as $n\to\infty$, and thus $t_{n}w_{n}\ge t_{\varepsilon}$ there for $n$ large.
	
	Using \eqref{eq3.7}, \eqref{eq3.9} and restricting both integrals in the convolution to $B_{\rho/n}$, we obtain
	\[
	\begin{aligned}
		t_{n}^{2}
		&\ge t_{n}^{2}\|w_{n}\|_{a}^{2}\\
		&=\int_{\R^{2}}\left(\frac{1}{|x|^{\mu}}*F(t_{n}w_{n})\right)
		f(t_{n}w_{n})\,t_{n}w_{n}\,dx\\
		&\ge \int_{B_{\rho/n}}\left(\int_{B_{\rho/n}}\frac{F(t_{n}w_{n})}{|x-y|^{\mu}}\,dy\right)
		f(t_{n}w_{n})\,t_{n}w_{n}\,dx\\
		&= t_{n}w_{n}f(t_{n}w_{n})F(t_{n}w_{n})
		\int_{B_{\rho/n}}\int_{B_{\rho/n}}\frac{1}{|x-y|^{\mu}}\,dx\,dy\\
		&\ge (\beta-\varepsilon)\,\mathrm{e}^{8\pi(t_{n}w_{n})^{2}}
		\int_{B_{\rho/n}}\int_{B_{\rho/n}}\frac{1}{|x-y|^{\mu}}\,dx\,dy.
	\end{aligned}
	\]
	
	Let $R=\rho/n$. For $x\in B_{R}(0)$ one has $B_{R-|x|}(0)\subset B_{R}(x)$, hence
	\[
	\begin{aligned}
		\int_{B_{R}}\int_{B_{R}}\frac{1}{|x-y|^{\mu}}\,dx\,dy
		&=\int_{B_{R}}dx\int_{B_{R}(x)}\frac{1}{|z|^{\mu}}\,dz\\
		&\ge \int_{B_{R}}dx\int_{B_{R-|x|}}\frac{1}{|z|^{\mu}}\,dz\\
		&=\frac{2\pi}{2-\mu}\int_{B_{R}}\bigl(R-|x|\bigr)^{2-\mu}\,dx\\
		&=\frac{4\pi^{2}}{2-\mu}\int_{0}^{R}(R-r)^{2-\mu}r\,dr\\
		&=\frac{4\pi^{2}}{(2-\mu)(3-\mu)(4-\mu)}\,R^{4-\mu}.
	\end{aligned}
	\]
	Setting
	\[
	D_{\mu}=\frac{4\pi^{2}}{(2-\mu)(3-\mu)(4-\mu)},
	\]
	we obtain
	\[
	\int_{B_{\rho/n}}\int_{B_{\rho/n}}\frac{1}{|x-y|^{\mu}}\,dx\,dy
	\ge D_{\mu}\left(\frac{\rho}{n}\right)^{4-\mu}.
	\]
	
	Moreover,
	\[
	8\pi(t_{n}w_{n})^{2}
	=8\pi t_{n}^{2}\frac{\log n}{2\pi\bigl(1+C_{s}+(a+C_{s})\delta_{n}\bigr)}
	=\frac{4t_{n}^{2}\log n}{1+C_{s}+(a+C_{s})\delta_{n}}.
	\]
	Hence
	\[
	t_{n}^{2}
	\ge (\beta-\varepsilon)D_{\mu}\rho^{4-\mu}
	\exp\left(
	\log n\left[
	\frac{4t_{n}^{2}}{1+C_{s}+(a+C_{s})\delta_{n}}-(4-\mu)
	\right]
	\right). 
	\]
	 this means $t_n$ is bounded.Thus, exists a constant $C_{1}>0$ such that
	\[
	\log n\left[
	\frac{4t_{n}^{2}}{1+C_{s}+(a+C_{s})\delta_{n}}-(4-\mu)
	\right]\le C_{1}
	\]
	for all $n$, which gives
	\begin{equation}\label{eq3.10}
		t_{n}^{2}
		\le \frac{4-\mu}{4}\Bigl(1+C_{s}+(a+C_{s})\delta_{n}\Bigr)
		+\frac{C_{2}}{\log n}
	\end{equation}
	for some constant $C_{2}>0$. Combining \eqref{eq3.8} and \eqref{eq3.10}, we obtain
	\begin{equation}\label{eq3.11}
		t_{n}^{2}
		=\frac{4-\mu}{4}\bigl(1+C_{s}\bigr)+o(1)
		\quad\text{as }n\to\infty.
	\end{equation}
	
	We now refine the lower bound on $t_{n}^{2}$. Set
	\[
	A_{n}=\{x\in B_{\rho} : t_{n}w_{n}(x)\ge t_{\varepsilon}\},
	\qquad
	B_{n}=B_{\rho}\setminus A_{n}.
	\]
	Since $w_n$ is supported in $B_\rho$, from \eqref{eq3.7} we have
	\[
	\begin{aligned}
		t_{n}^{2}
		&\ge t_{n}^{2}\|w_{n}\|_{a}^{2}\\
		&=\int_{B_{\rho}}\left(\frac{1}{|x|^{\mu}}*F(t_{n}w_{n})\right)
		f(t_{n}w_{n})\,t_{n}w_{n}\,dx\\
		&=\int_{A_{n}}\left(\frac{1}{|x|^{\mu}}*F(t_{n}w_{n})\right)
		f(t_{n}w_{n})\,t_{n}w_{n}\,dx\\
		&\quad+\int_{B_{n}}\left(\frac{1}{|x|^{\mu}}*F(t_{n}w_{n})\right)
		f(t_{n}w_{n})\,t_{n}w_{n}\,dx.
	\end{aligned}
	\]
	We claim that the contribution from $B_{n}$ tends to $0$ as $n\to\infty$. Indeed, by Lemma~\ref{Lem2.3} with $p=\frac{4}{4-\mu}$ and H\"older's inequality, there exists $C_{\mathrm{HLS}}>0$ such that
	\[
	\begin{aligned}
		\int_{B_{n}}\left(\frac{1}{|x|^{\mu}}*F(t_{n}w_{n})\right)
		f(t_{n}w_{n})\,t_{n}w_{n}\,dx
		&\le C_{\mathrm{HLS}}\,
		\|F(t_{n}w_{n})\|_{p}\,
		\|\chi_{B_{n}}t_{n}w_{n}f(t_{n}w_{n})\|_{p}.
	\end{aligned}
	\]
	By \eqref{eq3.10} and \eqref{eq3.4}, the sequence $\{t_{n}w_{n}\}$ is bounded in $H^{1}(\R^{2})$. Using \eqref{eq2.2}, Proposition~\ref{pro2.2} and the Sobolev embedding $H^{1}(\R^{2})\hookrightarrow L^{m}(\R^{2})$ for $m\ge2$ as in Lemma~\ref{Lem2.5}, we infer that $\|F(t_{n}w_{n})\|_{p}\le C$.
	
	Moreover, $w_{n}(x)\to0$ for a.e. $x\in B_{\rho}$ and $\{t_n\}$ is bounded, hence $t_{n}w_{n}(x)\to0$ for a.e. $x\in B_{\rho}$. On $B_{n}$ one has $|t_{n}w_{n}|\le t_{\varepsilon}$, so $\chi_{B_{n}}t_{n}w_{n}f(t_{n}w_{n})\to0$ a.e. in $B_{\rho}$, and it is dominated by a constant function in $L^{p}(B_{\rho})$. The dominated convergence theorem implies
	\[
	\|\chi_{B_{n}}t_{n}w_{n}f(t_{n}w_{n})\|_{p}\to0,
	\]
	so the integral over $B_{n}$ converges to $0$.
	
	Hence
	\[
	t_{n}^{2}
	\ge \int_{A_{n}}\left(\frac{1}{|x|^{\mu}}*F(t_{n}w_{n})\right)
	f(t_{n}w_{n})\,t_{n}w_{n}\,dx+o(1).
	\]
	On $A_{n}$ we have $t_{n}w_{n}\ge t_{\varepsilon}$, so by \eqref{eq3.9},
	\[
	t_{n}w_{n}f(t_{n}w_{n})F(t_{n}w_{n})
	\ge (\beta-\varepsilon)\,\mathrm{e}^{8\pi(t_{n}w_{n})^{2}}.
	\]
	Therefore,
	\[
	\begin{aligned}
		t_{n}^{2}
		&\ge \int_{B_{\rho/n}}\left(\int_{B_{\rho/n}}\frac{F(t_{n}w_{n})}{|x-y|^{\mu}}\,dy\right)
		f(t_{n}w_{n})\,t_{n}w_{n}\,dx+o(1)\\
		&= t_{n}w_{n}f(t_{n}w_{n})F(t_{n}w_{n})
		\int_{B_{\rho/n}}\int_{B_{\rho/n}}\frac{1}{|x-y|^{\mu}}\,dx\,dy+o(1)\\
		&\ge (\beta-\varepsilon)\,\mathrm{e}^{8\pi(t_{n}w_{n})^{2}}
		\int_{B_{\rho/n}}\int_{B_{\rho/n}}\frac{1}{|x-y|^{\mu}}\,dx\,dy+o(1).
	\end{aligned}
	\]
	Using again the estimate of the double integral and the expression of $w_{n}$ on $B_{\rho/n}$, we obtain
	\[
	t_{n}^{2}
	\ge (\beta-\varepsilon)D_{\mu}\rho^{4-\mu}
	\exp\left(
	\log n\left[
	\frac{4t_{n}^{2}}{1+C_{s}+(a+C_{s})\delta_{n}}-(4-\mu)
	\right]
	\right)+o(1).
	\]
	By \eqref{eq3.8},
	\[
	\frac{4t_{n}^{2}}{1+C_{s}+(a+C_{s})\delta_{n}}-(4-\mu)
	\ge (4-\mu)\left[\frac{1+C_{s}}{1+C_{s}+(a+C_{s})\delta_{n}}-1\right]
	=-(4-\mu)\frac{(a+C_{s})\delta_{n}}{1+C_{s}+(a+C_{s})\delta_{n}},
	\]
	and since $1+C_{s}+(a+C_{s})\delta_{n}\ge 1$ we obtain
	\[
	\frac{4t_{n}^{2}}{1+C_{s}+(a+C_{s})\delta_{n}}-(4-\mu)
	\ge -(4-\mu)(a+C_{s})\delta_{n}.
	\]
	Consequently,
	\[
	t_{n}^{2}
	\ge (\beta-\varepsilon) D_{\mu} \rho^{4-\mu}
	\exp\bigl(-(4-\mu)(a+C_{s})\,\delta_{n}\log n\bigr)+o(1).
	\]
	Since $\delta_{n}\log n=\frac{\rho^{2}}{4}+o(1)$, combining with \eqref{eq3.11} and letting $n\to\infty$ yields
	\[
	\frac{4-\mu}{4}\bigl(1+C_{s}\bigr)
	\ge (\beta-\varepsilon)D_{\mu}\rho^{4-\mu}
	\mathrm{e}^{-\frac{4-\mu}{4}(a+C_{s})\rho^{2}}.
	\]
	Since $\varepsilon>0$ is arbitrary,
	\[
	\beta
	\le \frac{(2-\mu)(3-\mu)(4-\mu)^{2}(1+C_{s})}{16\pi^{2}\rho^{4-\mu}}
	\mathrm{e}^{\frac{4-\mu}{4}(a+C_{s})\rho^{2}},
	\]
	which contradicts assumption $(f_{5})$. Therefore \eqref{eq3.5} holds for some $n$, and in particular
	\[
	c_{a}
	\le \max_{t\ge0}\mathcal{I}_{a}(t w_{n})
	<\frac{4-\mu}{8}\bigl(1+C_{s}\bigr),
	\]
	which proves \eqref{eq3.3}.
\end{proof}

\section{Ground state solution of the autonomous problem}

\begin{Lem}\label{Lem4.1}
	Assume that $(f_{1})$--$(f_{4})$ and $(f_{6})$ hold. Let $u_{n}\rightharpoonup u$ in $H^{1}(\R^{2})$ with $u_n\ge 0$ a.e. in $\R^2$, and assume that
	\begin{equation}\label{eq4.1}
		\int_{\R^{2}}\left(\frac{1}{|x|^{\mu}}*F(u_{n})\right) f(u_{n})u_{n}\,dx \le K_{0}
	\end{equation}
	for some constant $K_{0}>0$ and all $n$. Then for every $\phi\in C_{0}^{\infty}(\R^{2})$ we have
	\begin{equation}\label{eq4.2}
		\lim_{n\to\infty}\int_{\R^{2}}\left(\frac{1}{|x|^{\mu}}*F(u_{n})\right) f(u_{n})\phi\,dx
		=\int_{\R^{2}}\left(\frac{1}{|x|^{\mu}}*F(u)\right) f(u)\phi\,dx.
	\end{equation}
\end{Lem}

\begin{proof}
	Since $u_n\rightharpoonup u$ in $H^{1}(\R^{2})$, up to a subsequence $u_{n}\to u$ a.e. in $\R^{2}$ and $u\ge 0$ a.e. By $(f_{3})$ we have $F(t)\ge 0$ and $f(t)t\ge 0$ for $t\ge 0$. Writing
	\[
	\int_{\R^{2}}\left(\frac{1}{|x|^{\mu}}*F(u_{n})\right) f(u_{n})u_{n}\,dx
	=\int_{\R^{2}}\int_{\R^{2}}\frac{F(u_n(y))\,f(u_n(x))u_n(x)}{|x-y|^\mu}\,dy\,dx,
	\]
	Fatou's lemma on $\R^{2}\times\R^{2}$ and \eqref{eq4.1} yield
	\begin{equation}\label{eq4.3}
		\int_{\R^{2}}\left(\frac{1}{|x|^{\mu}}*F(u)\right) f(u)u\,dx \le K_{0}.
	\end{equation}
	
	Let $\Omega=\operatorname{supp}\phi$ and fix $\varepsilon>0$. Set
	\[
	M_{\varepsilon}
	=\frac{2K_{0}\|\phi\|_{\infty}}{\varepsilon}.
	\]
	Then, for every $n$,
	\begin{equation}\label{eq4.4}
		\begin{aligned}
			\int_{\{u_{n}\ge M_{\varepsilon}\}}
			\left(\frac{1}{|x|^{\mu}}*F(u_{n})\right)
			\bigl|f(u_{n})\phi\bigr|\,dx
			&\le \frac{\|\phi\|_{\infty}}{M_{\varepsilon}}
			\int_{\R^{2}}\left(\frac{1}{|x|^{\mu}}*F(u_{n})\right)
			f(u_{n})u_{n}\,dx\\
			&\le \frac{\varepsilon}{2K_{0}}
			\int_{\R^{2}}\left(\frac{1}{|x|^{\mu}}*F(u_{n})\right)
			f(u_{n})u_{n}\,dx\\
			&\le \frac{\varepsilon}{2},
		\end{aligned}
	\end{equation}
	and similarly, using \eqref{eq4.3},
	\begin{equation}\label{eq4.5}
		\int_{\{u\ge M_{\varepsilon}\}}
		\left(\frac{1}{|x|^{\mu}}*F(u)\right)|f(u)\phi|\,dx
		\le \frac{\varepsilon}{2}.
	\end{equation}
	Hence the contribution to \eqref{eq4.2} coming from the sets
	\[
	\{u_{n}\ge M_{\varepsilon}\}\cup\{u\ge M_{\varepsilon}\}
	\]
	is bounded by $\varepsilon$ for all $n$.
	
	Define
	\[
	g_{n}=f(u_{n})\,\phi\,\chi_{\{u_{n}\le M_{\varepsilon}\}},
	\qquad
	g=f(u)\,\phi\,\chi_{\{u\le M_{\varepsilon}\}}.
	\]
	By continuity of $f$ and $u_n\to u$ a.e., we have $g_n\to g$ a.e. in $\Omega$ and
	\[
	|g_n|\le \|\phi\|_\infty \max_{0\le t\le M_\varepsilon} |f(t)| \,\chi_\Omega,
	\]
	so
	\begin{equation}\label{eq4.6}
		g_n\to g \quad\text{in }L^{\frac{4}{4-\mu}}(\R^{2}).
	\end{equation}
	
	We next control the contribution of large values inside the convolution. By $(f_{6})$ there exist $M_{0}>0$ and $t_{0}>0$ such that
	\[
	F(t)\le M_{0} f(t)
	\quad\text{for all }t\ge t_{0}.
	\]
	Choose $K_{\varepsilon}>\max\{t_{0},M_{\varepsilon}\}$ so large that
	\begin{equation}\label{eq4.7}
		C_{\mu}^{\frac12}\,\|\phi\|_{\infty}\,|\Omega|^{\frac{4-\mu}{4}}
		\Bigl(\max_{0\le t\le M_{\varepsilon}}|f(t)|\Bigr)
		\left(\frac{M_{0}K_{0}}{K_{\varepsilon}}\right)^{\frac{1}{2}}
		<\varepsilon,
	\end{equation}
	where $C_{\mu}$ is the constant in Lemma~\ref{Lem2.3}.
	
	Set
	\[
	F_{n}^{\mathrm{tail}}=F(u_{n})\,\chi_{\{u_{n}\ge K_{\varepsilon}\}}.
	\]
	Using Lemma~\ref{Lem2.4} with $f=F_{n}^{\mathrm{tail}}$ and $h=|g_n|$, and then Lemma~\ref{Lem2.3}, we obtain
	\begin{equation}\label{eq4.8}
		\begin{aligned}
			\int_{\R^{2}}
			\left(\frac{1}{|x|^{\mu}}*F_{n}^{\mathrm{tail}}\right)
			|g_{n}|\,dx
			&\le
			\left(\int_{\R^{2}}\left(\frac{1}{|x|^{\mu}}*F_{n}^{\mathrm{tail}}\right)
			F_{n}^{\mathrm{tail}}\,dx\right)^{\frac{1}{2}}
			\left(\int_{\R^{2}}\left(\frac{1}{|x|^{\mu}}*|g_n|\right)|g_n|\,dx\right)^{\frac{1}{2}}\\
			&\le
			\left(\int_{\R^{2}}\left(\frac{1}{|x|^{\mu}}*F_{n}^{\mathrm{tail}}\right)
			F_{n}^{\mathrm{tail}}\,dx\right)^{\frac{1}{2}}
			C_{\mu}^{\frac12}\left(\int_{\R^{2}}|g_n|^{\frac{4}{4-\mu}}\,dx\right)^{\frac{4-\mu}{4}}.
		\end{aligned}
	\end{equation}
	Moreover, since $F\ge 0$ and the kernel is positive,
	\[
	\left(\frac{1}{|x|^\mu}*F_n^{\mathrm{tail}}\right)\le \left(\frac{1}{|x|^\mu}*F(u_n)\right),
	\]
	hence
	\[
	\int_{\R^{2}}\left(\frac{1}{|x|^{\mu}}*F_{n}^{\mathrm{tail}}\right)F_{n}^{\mathrm{tail}}\,dx
	\le \int_{\{u_n\ge K_\varepsilon\}}\left(\frac{1}{|x|^{\mu}}*F(u_{n})\right)F(u_{n})\,dx.
	\]
	On $\{u_n\ge K_\varepsilon\}$, using $(f_6)$ and $u_n\ge K_\varepsilon$ we have
	\[
	F(u_n)\le M_0 f(u_n)\le \frac{M_0}{K_\varepsilon} f(u_n)u_n,
	\]
	therefore, by \eqref{eq4.1},
	\[
	\int_{\R^{2}}\left(\frac{1}{|x|^{\mu}}*F_{n}^{\mathrm{tail}}\right)F_{n}^{\mathrm{tail}}\,dx
	\le \frac{M_{0}}{K_{\varepsilon}}
	\int_{\R^{2}}
	\left(\frac{1}{|x|^{\mu}}*F(u_{n})\right)f(u_{n})u_{n}\,dx
	\le \frac{M_{0}K_{0}}{K_{\varepsilon}}.
	\]
	Combining this with \eqref{eq4.8} and \eqref{eq4.7}, and using the bound
	\[
	\left(\int_{\R^2}|g_n|^{\frac{4}{4-\mu}}\,dx\right)^{\frac{4-\mu}{4}}
	\le \|\phi\|_{\infty}\,|\Omega|^{\frac{4-\mu}{4}}
	\Bigl(\max_{0\le t\le M_{\varepsilon}}|f(t)|\Bigr),
	\]
	we obtain, for all $n$,
	\begin{equation}\label{eq4.9}
		\left|
		\int_{\R^{2}}
		\left(\frac{1}{|x|^{\mu}}*F(u_n)\right) g_n\,dx
		-\int_{\R^{2}}
		\left(\frac{1}{|x|^{\mu}}*\bigl(F(u_n)\chi_{\{u_n\le K_{\varepsilon}\}}\bigr)\right) g_n\,dx
		\right|
		\le \varepsilon.
	\end{equation}
	Similarly, setting $F^{\mathrm{tail}}=F(u)\chi_{\{u\ge K_\varepsilon\}}$ and using \eqref{eq4.3} in the same argument, we also have
	\begin{equation}\label{eq4.10}
		\left|
		\int_{\R^{2}}
		\left(\frac{1}{|x|^{\mu}}*F(u)\right) g\,dx
		-\int_{\R^{2}}
		\left(\frac{1}{|x|^{\mu}}*\bigl(F(u)\chi_{\{u\le K_{\varepsilon}\}}\bigr)\right) g\,dx
		\right|
		\le \varepsilon.
	\end{equation}
	
	Now set
	\[
	F_{n}^{\mathrm{tr}}=F(u_{n})\,\chi_{\{u_{n}\le K_{\varepsilon}\}},
	\qquad
	F^{\mathrm{tr}}=F(u)\,\chi_{\{u\le K_{\varepsilon}\}}.
	\]
	For $0\le t\le K_\varepsilon$ define $H(t)=F(t)\,t^{-\frac{4-\mu}{2}}$ for $t>0$ and $H(0)=0$. By $(f_1)$ one has $H(t)\to 0$ as $t\to 0^{+}$, hence $H$ is bounded on $[0,K_\varepsilon]$ and there exists $C_\varepsilon>0$ such that
	\begin{equation}\label{eq4.11}
		F(t)\le C_\varepsilon\,t^{\frac{4-\mu}{2}}
		\quad\text{for all }t\in[0,K_\varepsilon].
	\end{equation}
	In particular,
	\[
	\int_{\R^2}\bigl|F_n^{\mathrm{tr}}\bigr|^{\frac{4}{4-\mu}}\,dx
	\le C_\varepsilon \int_{\R^2} |u_n|^{2}\,dx,
	\]
	so $\{F_n^{\mathrm{tr}}\}$ is bounded in $L^{\frac{4}{4-\mu}}(\R^2)$.
	
	Define
	\[
	\zeta_{n}(x)=\left(\frac{1}{|x|^{\mu}}*F_{n}^{\mathrm{tr}}\right)(x),
	\qquad
	\zeta(x)=\left(\frac{1}{|x|^{\mu}}*F^{\mathrm{tr}}\right)(x).
	\]
	Fix $x\in\Omega$ and $R>1$. Since $\mu<2$ and $|F_n^{\mathrm{tr}}|\le \max_{0\le t\le K_\varepsilon} |F(t)|$, the function $|x-y|^{-\mu}$ is integrable on $B_R(x)$ and dominated convergence yields
	\[
	\int_{B_R(x)}\frac{|F_n^{\mathrm{tr}}(y)-F^{\mathrm{tr}}(y)|}{|x-y|^\mu}\,dy \to 0
	\quad\text{as }n\to\infty.
	\]
	For the complement $\R^2\setminus B_R(x)$, let $p=\frac{4}{4-\mu}$ and $p'=\frac{4}{\mu}$. By H\"older's inequality and the boundedness of $\{F_n^{\mathrm{tr}}\}$ in $L^p(\R^2)$,
	\[
	\int_{\R^2\setminus B_R(x)}\frac{|F_n^{\mathrm{tr}}(y)|}{|x-y|^\mu}\,dy
	\le \|F_n^{\mathrm{tr}}\|_{p}
	\left(\int_{|x-y|>R}\frac{1}{|x-y|^{\mu p'}}\,dy\right)^{\frac{1}{p'}}
	\le C\,R^{-\frac{\mu}{2}},
	\]
	and the same estimate holds with $F_n^{\mathrm{tr}}$ replaced by $F^{\mathrm{tr}}$. Letting first $n\to\infty$ and then $R\to\infty$, we obtain
	\[
	\zeta_n(x)\to \zeta(x)\quad\text{for every }x\in\Omega.
	\]
	Moreover, taking $R=1$ in the above decomposition yields $|\zeta_n(x)|\le C$ for all $x\in\Omega$ and all $n$, with $C$ independent of $n$.
	
	Since $u_n\to u$ a.e. in $\Omega$, we have $g_n(x)\to g(x)$ for a.e. $x\in\Omega$ and $|g_n|\le C\,\chi_\Omega$. Therefore, by dominated convergence,
	\begin{equation}\label{eq4.12}
		\int_{\R^{2}} \zeta_n(x)\,g_n(x)\,dx
		\longrightarrow
		\int_{\R^{2}} \zeta(x)\,g(x)\,dx.
	\end{equation}
	
	Finally, we split
	\[
	\int_{\R^{2}}\left(\frac{1}{|x|^{\mu}}*F(u_{n})\right) f(u_{n})\phi\,dx
	=\int_{\{u_n\ge M_\varepsilon\}}\left(\frac{1}{|x|^{\mu}}*F(u_{n})\right) f(u_{n})\phi\,dx
	+\int_{\R^{2}}\left(\frac{1}{|x|^{\mu}}*F(u_{n})\right) g_n\,dx,
	\]
	and similarly for $u$. Using \eqref{eq4.4}, \eqref{eq4.5}, \eqref{eq4.9}, \eqref{eq4.10} and \eqref{eq4.12}, and recalling that $\varepsilon>0$ was arbitrary, we conclude that
	\[
	\int_{\R^{2}}\left(\frac{1}{|x|^{\mu}}*F(u_{n})\right) f(u_{n})\phi\,dx
	\longrightarrow
	\int_{\R^{2}}\left(\frac{1}{|x|^{\mu}}*F(u)\right) f(u)\phi\,dx,
	\]
	which is \eqref{eq4.2}.
\end{proof}

\begin{Lem}\label{Lem4.2}
	Assume that $(f_{1})$--$(f_{4})$ and $(f_{6})$ hold. Let $\{u_{n}\}$ be a $(PS)_{c_{a}}$ sequence for $\mathcal{I}_{a}$ with
	\[
	c_{a}<\frac{4-\mu}{8}\bigl(1+C_s\bigr).
	\]
	Then the following conclusions hold:
	\begin{itemize}
		\item[(i)] $\{u_{n}\}$ is bounded in $W_{a}$, and up to a subsequence $u_{n}\rightharpoonup u$ for some $u\in W_{a}$;
		\item[(ii)] $u\ge0$ in $\R^{2}$;
		\item[(iii)] $\mathcal{I}_{a}'(u)=0$.
	\end{itemize}
\end{Lem}

\begin{proof}
	We use the convention that $f(t)=0$ for $t\le0$, hence $F(t)=0$ for $t\le0$.
	Since $\{u_{n}\}$ is a $(PS)_{c_{a}}$ sequence, we have
	\[
	\mathcal{I}_{a}(u_{n})\to c_{a}
	\quad\text{and}\quad
	\|\mathcal{I}_{a}'(u_{n})\|_{W_{a}^{*}}\to 0.
	\]
	By the definition of $\mathcal{I}_{a}$,
	\[
	\begin{aligned}
		\mathcal{I}_{a}(u_{n})
		&=\frac{1}{2}\|u_{n}\|_{a}^{2}
		-\frac{1}{2}\int_{\R^{2}}
		\left(\frac{1}{|x|^{\mu}}*F(u_{n})\right)F(u_{n})\,dx,\\
		\bigl\langle \mathcal{I}_{a}'(u_{n}),u_{n}\bigr\rangle
		&=\|u_{n}\|_{a}^{2}
		-\int_{\R^{2}}
		\left(\frac{1}{|x|^{\mu}}*F(u_{n})\right)f(u_{n})u_{n}\,dx.
	\end{aligned}
	\]
	
	\noindent (i) By $(f_{3})$, for $t\ge0$ one has $f(t)t\ge\theta F(t)$ with $\theta>1$, hence
	\[
	\frac{1}{2\theta}f(t)t-\frac{1}{2}F(t)\ge 0
	\quad\text{for all }t\in\R.
	\]
	Using this and the nonnegativity of the kernel, we compute
	\[
	\begin{aligned}
		\mathcal{I}_{a}(u_{n})
		-\frac{1}{2\theta}\bigl\langle \mathcal{I}_{a}'(u_{n}),u_{n}\bigr\rangle
		&=\left(\frac{1}{2}-\frac{1}{2\theta}\right)\|u_{n}\|_{a}^{2}
		+\int_{\R^{2}}
		\left(\frac{1}{|x|^{\mu}}*F(u_{n})\right)
		\left(\frac{1}{2\theta}f(u_{n})u_{n}
		-\frac{1}{2}F(u_{n})\right)\,dx\\
		&\ge \left(\frac{1}{2}-\frac{1}{2\theta}\right)\|u_{n}\|_{a}^{2}.
	\end{aligned}
	\]
	Hence
	\[
	c_a+o_n(1)
	\ge \left(\frac{1}{2}-\frac{1}{2\theta}\right)\|u_{n}\|_{a}^{2}.
	\]
	 this implies that $\{\|u_n\|_a\}$ is bounded. Therefore, up to a subsequence,
	\[
	u_{n}\rightharpoonup u \text{ in }W_{a},\quad
	u_{n}\to u \text{ in }L^{p}_{\mathrm{loc}}(\R^{2})\text{ for all }p\in[1,\infty),\quad
	u_{n}\to u \text{ a.e. in }\R^{2}.
	\]
	
	\noindent (ii) Let $u_{n}^{-}=\max\{-u_{n},0\}$ and $u_{n}^{+}=\max\{u_{n},0\}$. By the convention $f(t)=0$ for $t\le0$ we have
	\[
	f(u_{n})u_{n}^{-}=0 \quad\text{a.e. in }\R^{2}.
	\]
	Taking $\varphi=u_{n}^{-}$ in $\langle \mathcal{I}_{a}'(u_{n}),\varphi\rangle$ yields
	\[
	\begin{aligned}
		\bigl\langle \mathcal{I}_{a}'(u_{n}),u_{n}^{-}\bigr\rangle
		&=\int_{\R^{2}}\nabla u_{n}\cdot\nabla u_{n}^{-}\,dx
		+\frac{1}{2}\int_{\R^{2}}\int_{\R^{2}}
		\frac{(u_{n}(x)-u_{n}(y))(u_{n}^{-}(x)-u_{n}^{-}(y))}{|x-y|^{2+2s}}\,dx\,dy
		+\int_{\R^{2}}a\,u_{n}u_{n}^{-}\,dx.
	\end{aligned}
	\]
	Let  $r^{-}=\max\{-r,0\}$,We use the pointwise inequality
	\[
	(r-s)(r^{-}-s^{-})\le -(r^{-}-s^{-})^{2}
	\quad\text{for all }r,s\in\R,
	\]
	which gives
	\[
	\frac{1}{2}\int_{\R^{2}}\int_{\R^{2}}
	\frac{(u_{n}(x)-u_{n}(y))(u_{n}^{-}(x)-u_{n}^{-}(y))}{|x-y|^{2+2s}}\,dx\,dy
	\le -\frac{1}{2}\int_{\R^{2}}\int_{\R^{2}}
	\frac{|u_{n}^{-}(x)-u_{n}^{-}(y)|^{2}}{|x-y|^{2+2s}}\,dx\,dy.
	\]
	Moreover,
	\[
	\int_{\R^{2}}\nabla u_{n}\cdot\nabla u_{n}^{-}\,dx
	=-\int_{\R^{2}}|\nabla u_{n}^{-}|^{2}\,dx,
	\qquad
	\int_{\R^{2}}a\,u_{n}u_{n}^{-}\,dx
	=-\int_{\R^{2}}a\,(u_{n}^{-})^{2}\,dx.
	\]
	Therefore
	\[
	\bigl\langle \mathcal{I}_{a}'(u_{n}),u_{n}^{-}\bigr\rangle
	\le -\|u_{n}^{-}\|_{a}^{2}.
	\]
	Since $\|\mathcal{I}_{a}'(u_{n})\|_{W_a^*}\to0$,
	\[
	\left|\bigl\langle \mathcal{I}_{a}'(u_{n}),u_{n}^{-}\bigr\rangle\right|
	\le \|\mathcal{I}_{a}'(u_{n})\|_{W_a^*}\,\|u_{n}^{-}\|_{a},
	\]
	hence $\|u_{n}^{-}\|_{a}\to0$. In particular, $u_{n}^{+}=u_{n}+u_{n}^{-}\rightharpoonup u$ in $W_a$, and $u\ge0$ a.e. in $\R^{2}$.
	
	\noindent (iii) We prove that $u$ is a critical point of $\mathcal{I}_{a}$. Since $F(t)=0$ and $f(t)=0$ for $t\le0$, we have
	\[
		F(u_n)
		=
		\begin{cases}
			F(u_n^+), & u_n >0,\\[4pt]
			0, & u_n\leq0;
		\end{cases}
	\qquad 
		f(u_n)
		=
		\begin{cases}
			f(u_n^+), & u_n >0,\\[4pt]
			0, & u_n\leq0;
		\end{cases} 
	\]
   
	so the nonlinear terms in $\mathcal{I}_a$ and $\mathcal{I}_a'$ are unchanged by replacing $u_n$ with $u_n^+$. Moreover,
	\[
	\|u_n-u_n^+\|_a=\|u_n^-\|_a\to0,
	\]
	which yields
	\[
	\mathcal{I}_{a}(u_{n}^{+})-\mathcal{I}_{a}(u_{n})\to0,
	\qquad
	\|\mathcal{I}_{a}'(u_{n}^{+})-\mathcal{I}_{a}'(u_{n})\|_{W_a^*}\to0.
	\]
	Thus $\{u_n^+\}$ is still a $(PS)_{c_a}$ sequence. Replacing $u_n$ by $u_n^+$, we may assume $u_n\ge0$ for all $n$.
	
	For every $\varphi\in C_{c}^{\infty}(\R^{2})$ we have
	\[
	\begin{aligned}
		\bigl\langle \mathcal{I}_{a}'(u_{n}),\varphi\bigr\rangle
		&=\int_{\R^{2}}\nabla u_{n}\cdot\nabla\varphi\,dx
		+\frac{1}{2}\int_{\R^{2}}\int_{\R^{2}}
		\frac{(u_{n}(x)-u_{n}(y))(\varphi(x)-\varphi(y))}{|x-y|^{2+2s}}\,dx\,dy
		+\int_{\R^{2}}a\,u_{n}\varphi\,dx\\
		&\quad-\int_{\R^{2}}
		\left(\frac{1}{|x|^{\mu}}*F(u_{n})\right)f(u_{n})\varphi\,dx
		\longrightarrow 0.
	\end{aligned}
	\]
	The first three terms converge by weak convergence in $W_a$. It remains to show that
	\begin{equation}\label{eq4.13}
		\int_{\R^{2}}\left(\frac{1}{|x|^{\mu}}*F(u_{n})\right)f(u_{n})\varphi\,dx
		\longrightarrow
		\int_{\R^{2}}\left(\frac{1}{|x|^{\mu}}*F(u)\right)f(u)\varphi\,dx
	\end{equation}
	for all $\varphi\in C_{c}^{\infty}(\R^{2})$.
	
	Since $\|\mathcal{I}_{a}'(u_{n})\|_{W_a^*}\to0$, we have
	\[
	\left|\bigl\langle \mathcal{I}_{a}'(u_{n}),u_{n}\bigr\rangle\right|
	\le \|\mathcal{I}_{a}'(u_{n})\|_{W_a^*}\,\|u_n\|_a,
	\]
	so
	\[
	\int_{\R^{2}}
	\left(\frac{1}{|x|^{\mu}}*F(u_{n})\right)f(u_{n})u_{n}\,dx
	=\|u_{n}\|_{a}^{2}-\bigl\langle \mathcal{I}_{a}'(u_{n}),u_{n}\bigr\rangle.
	\]
	Using the boundedness of $\|u_{n}\|_{a}$, there exists $C>0$ such that
	\begin{equation}\label{eq4.14}
		\int_{\R^{2}}
		\left(\frac{1}{|x|^{\mu}}*F(u_{n})\right)f(u_{n})u_{n}\,dx
		\le C
	\end{equation}
	for all $n$. Thus $\{u_{n}\}$ satisfies the assumptions of Lemma~\ref{Lem4.1}, and \eqref{eq4.13} follows.
	Passing to the limit in $\langle \mathcal{I}_{a}'(u_{n}),\varphi\rangle$ we obtain $\langle \mathcal{I}_{a}'(u),\varphi\rangle=0$ for all $\varphi\in C_c^\infty(\R^2)$, and by density of $C_c^\infty(\R^2)$ in $W_a$ we conclude $\mathcal{I}_{a}'(u)=0$ in $W_a^*$.
\end{proof}

Now we prove the existence result for the autonomous problem \eqref{eq3.1}.

\begin{Thm}\label{Thm4.3}
	Assume that $(f_{1})$--$(f_{6})$ hold. Then for any $a>0$, problem \eqref{eq3.1} admits a positive ground state solution.
\end{Thm}

\begin{proof}
	Arguing as in Lemma~\ref{Lem2.6}, one sees that $\mathcal{I}_{a}$ has the mountain pass geometry. Hence there exists a $(PS)_{c_{a}}$ sequence $\{u_{n}\}\subset W_{a}$ such that
	\[
	\mathcal{I}_{a}(u_{n})\to c_{a},\qquad
	\mathcal{I}_{a}'(u_{n})\to0\ \text{in }W_{a}^{*},
	\]
	where $c_{a}>0$ is the mountain pass level. Moreover, by Lemma~\ref{Lem3.1},
	\[
	c_{a}<\frac{4-\mu}{8}\bigl(1+C_s\bigr).
	\]
	Applying Lemma~\ref{Lem4.2}, up to a subsequence we have
	\[
	u_{n}\rightharpoonup u\ge0 \ \text{in }W_{a},\qquad
	\mathcal{I}_{a}'(u)=0,
	\]
	and, up to replacing $u_n$ by $u_n^{+}$, we may assume $u_{n}\ge0$ for all $n$.
	
	\textbf{Step 1}
	We show that $\{u_{n}\}$ cannot vanish in the sense of Lions. Suppose by contradiction that for some $r>0$,
	\begin{equation}\label{eq4.15}
		\lim_{n\to\infty}\sup_{y\in\R^{2}}
		\int_{B_{r}(y)}|u_{n}|^{2}\,dx=0.
	\end{equation}
	Then, by Lions' concentration--compactness lemma (see \cite{1996Willem}), it follows that
	\[
	u_{n}\to0\quad\text{in }L^{p}(\R^{2}),\quad 2<p<\infty.
	\]
	
	We claim that
	\begin{equation}\label{eq4.16}
		\int_{\R^{2}}
		\left(\frac{1}{|x|^{\mu}}*F(u_{n})\right)F(u_{n})\,dx\longrightarrow 0
		\quad\text{as }n\to\infty.
	\end{equation}
	From $(f_{3})$ and \eqref{eq4.14} we have
	\begin{equation}\label{eq4.17}
		\theta\int_{\R^{2}}
		\left(\frac{1}{|x|^{\mu}}*F(u_{n})\right)F(u_{n})\,dx
		\le\int_{\R^{2}}
		\left(\frac{1}{|x|^{\mu}}*F(u_{n})\right)f(u_{n})u_{n}\,dx
		\le C
	\end{equation}
	for some $C>0$ independent of $n$.
	
	Fix $\varepsilon>0$. By $(f_{6})$ there exist $M_{0}>0$ and $t_{0}>0$ such that
	\[
	F(t)\le M_{0}f(t)\quad\text{for }t\ge t_{0}.
	\]
	Choose $M_{\varepsilon}>\max\{t_{0},M_{0}C/\varepsilon\}$. Using $(f_{6})$ and \eqref{eq4.17}, we obtain
	\begin{equation}\label{eq4.18}
		\int_{\{u_{n}\ge M_{\varepsilon}\}}
		\left(\frac{1}{|x|^{\mu}}*F(u_{n})\right)F(u_{n})\,dx
		\le \varepsilon.
	\end{equation}
	
	Next, using $(f_{1})$ and the continuity of $f$ and $F$ at $0$, for the same $\varepsilon>0$ we can choose $N_{\varepsilon}\in(0,1)$ such that
	\[
	|F(t)|\le\varepsilon |t|^{\frac{4-\mu}{2}}
	\quad\text{and}\quad
	|f(t)t|\le\varepsilon |t|^{\frac{4-\mu}{2}}
	\qquad\text{for }|t|\le N_{\varepsilon}.
	\]
	Then, by $(f_{3})$, Lemmas~\ref{Lem2.3}--\ref{Lem2.4}, and \eqref{eq4.17},
	\begin{equation}\label{eq4.19}
		\begin{aligned}
			&\int_{\{u_{n}\le N_{\varepsilon}\}}
			\left(\frac{1}{|x|^{\mu}}*F(u_{n})\right)F(u_{n})\,dx\\
			&\le \frac{1}{\theta}
			\int_{\{u_{n}\le N_{\varepsilon}\}}
			\left(\frac{1}{|x|^{\mu}}*F(u_{n})\right)f(u_{n})u_{n}\,dx\\
			&\le \frac{\varepsilon}{\theta}
			\int_{\{u_{n}\le N_{\varepsilon}\}}
			\left(\frac{1}{|x|^{\mu}}*F(u_{n})\right)u_{n}^{\frac{4-\mu}{2}}\,dx\\
			&\le \frac{\varepsilon}{\theta}
			\left(\int_{\R^{2}}
			\left(\frac{1}{|x|^{\mu}}*F(u_{n})\right)F(u_{n})\,dx\right)^{\frac12}
			\left(\int_{\R^{2}}
			\left(\frac{1}{|x|^{\mu}}*u_{n}^{\frac{4-\mu}{2}}\right)
			u_{n}^{\frac{4-\mu}{2}}\,dx\right)^{\frac12}\\
			&\le C\,\varepsilon,
		\end{aligned}
	\end{equation}
	where $C>0$ is independent of $n$.
	
	On the intermediate set $\{N_{\varepsilon}\le u_{n}\le M_{\varepsilon}\}$, since $F$ is continuous, there exists $C_\varepsilon>0$ such that $|F(t)|\le C_\varepsilon$ for $t\in[N_\varepsilon,M_\varepsilon]$. Hence, using Lemma~\ref{Lem2.3} with $p=\frac{4}{4-\mu}$,
	\begin{equation}\label{eq4.20}
		\begin{aligned}
			\int_{\{N_{\varepsilon}\le u_{n}\le M_{\varepsilon}\}}
			\left(\frac{1}{|x|^{\mu}}*F(u_{n})\right)F(u_{n})\,dx
			&\le C_{\mu}\,\|F(u_n)\|_{\frac{4}{4-\mu}}\,\|F(u_n)\chi_{\{N_\varepsilon \le u_n\leq M_\varepsilon\}}\|_{\frac{4}{4-\mu}}\\
			&\le C\,|\{u_n\ge N_\varepsilon\}|^{\frac{4-\mu}{4}}
			\longrightarrow 0,
		\end{aligned}
	\end{equation}
	where we used that $|\{u_n\ge N_\varepsilon\}|\le N_\varepsilon^{-4}\|u_n\|_4^{4}\to0$ because $u_n\to0$ in $L^{4}(\R^{2})$.
	Combining \eqref{eq4.18}, \eqref{eq4.19}, and \eqref{eq4.20}, and using the arbitrariness of $\varepsilon>0$, we obtain \eqref{eq4.16}.
	
	Since $\{u_{n}\}$ is a $(PS)_{c_{a}}$ sequence, by \eqref{eq4.16} we have
	\begin{equation}\label{eq4.21}
		c_{a}
		=\lim_{n\to\infty}\mathcal{I}_{a}(u_{n})
		=\frac12\lim_{n\to\infty}\|u_{n}\|_{a}^{2}.
	\end{equation}
	Therefore
	\[
	\lim_{n\to\infty}\|u_{n}\|_{a}^{2}
	=2c_{a}<\frac{4-\mu}{4}(1+C_s).
	\]
	Thus there exist $\delta\in(0,1)$ and $n_{0}\in\N$ such that
	\begin{equation}\label{eq4.22}
		\|u_{n}\|_{a}^{2}
		\le\frac{4-\mu}{4}(1+C_s)(1-\delta),
		\qquad n\ge n_{0}.
	\end{equation}
	
	Using the Hardy--Littlewood--Sobolev inequality and $(f_{3})$ we have
	\[
	\int_{\R^{2}}
	\left(\frac{1}{|x|^{\mu}}*F(u_{n})\right)f(u_{n})u_{n}\,dx
	\le C\,\|F(u_{n})\|_{\frac{4}{4-\mu}}
	\|f(u_{n})u_{n}\|_{\frac{4}{4-\mu}},
	\qquad
	\|F(u_{n})\|_{\frac{4}{4-\mu}}
	\le\|f(u_{n})u_{n}\|_{\frac{4}{4-\mu}}.
	\]
	By Lemma~\ref{Lem2.3} and \eqref{eq4.16}, $\|F(u_{n})\|_{\frac{4}{4-\mu}}\to0$. Moreover, by $(f_{1})$--$(f_{2})$ and the uniform bound \eqref{eq4.22}, arguing as in Lemma~\ref{Lem2.5} one checks that $\{\|f(u_n)u_n\|_{\frac{4}{4-\mu}}\}$ is bounded. Hence
	\[
	\int_{\R^{2}}
	\left(\frac{1}{|x|^{\mu}}*F(u_{n})\right)f(u_{n})u_{n}\,dx
	\longrightarrow 0.
	\]
	On the other hand,
	\[
	\bigl\langle \mathcal{I}_{a}'(u_{n}),u_{n}\bigr\rangle
	=\|u_{n}\|_{a}^{2}
	-\int_{\R^{2}}
	\left(\frac{1}{|x|^{\mu}}*F(u_{n})\right)f(u_{n})u_{n}\,dx
	\longrightarrow 0,
	\]
	so $\|u_{n}\|_{a}\to0$. Together with \eqref{eq4.16} this yields $\mathcal{I}_{a}(u_n)\to0$, hence $c_a=0$, a contradiction. Therefore vanishing cannot occur, and there exist $r>0$, $\eta_{0}>0$ and a sequence $\{y_{n}\}\subset\R^{2}$ such that
	\[
	\liminf_{n\to\infty}
	\int_{B_{r}(y_{n})}u_{n}^{2}\,dx\ge\eta_{0}>0.
	\]
	
	\textbf{Step 2}
	Define
	\[
	v_{n}(x)=u_{n}(x+y_{n})\ge0.
	\]
	Since $\mathcal{I}_{a}$ is translation invariant, $\{v_{n}\}$ is again a $(PS)_{c_{a}}$ sequence. Up to a subsequence,
	\[
	v_{n}\rightharpoonup v\ge0\ \text{ in }W_{a},\qquad
	\mathcal{I}_{a}'(v)=0,
	\]
	and $v_n\to v$ in $L^{2}(B_r(0))$. Therefore
	\[
	\int_{B_{r}(0)}v^{2}\,dx
	=\lim_{n\to\infty}\int_{B_{r}(0)}v_{n}^{2}\,dx
	=\lim_{n\to\infty}\int_{B_{r}(y_{n})}u_{n}^{2}\,dx
	\ge\eta_{0}>0,
	\]
	so $v\ne0$.
	
	Since $\mathcal{I}_{a}'(v)=0$ and $v\ne0$, we have $v\in\mathcal{N}_{a}$, hence
	\[
	c_{a}\le\mathcal{I}_{a}(v).
	\]
	On the other hand, up to a subsequence,
    \[
	f(v_n)v_n-F(v_n)\to f(v)v-F(v)\quad\text{for a.e. }x\in\R^{2}.
    \]
	Since $f(t)t-F(t)\ge0$ for $t\ge0$ by $(f_{3})$, Fatou's lemma yields
	\[
	\begin{aligned}
		c_{a}
		&=\lim_{n\to\infty}
		\Bigl(\mathcal{I}_{a}(v_{n})
		-\frac12\bigl\langle \mathcal{I}_{a}'(v_{n}),v_{n}\bigr\rangle\Bigr)\\
		&=\frac12\liminf_{n\to\infty}
		\int_{\R^{2}}
		\left(\frac{1}{|x|^{\mu}}*F(v_{n})\right)
		\bigl(f(v_{n})v_{n}-F(v_{n})\bigr)\,dx\\
		&\ge\frac12
		\int_{\R^{2}}
		\left(\frac{1}{|x|^{\mu}}*F(v)\right)
		\bigl(f(v)v-F(v)\bigr)\,dx\\
		&=\mathcal{I}_{a}(v)-\frac12\bigl\langle \mathcal{I}_{a}'(v),v\bigr\rangle
		=\mathcal{I}_{a}(v).
	\end{aligned}
	\]
	Therefore $\mathcal{I}_{a}(v)=c_{a}$ and $v$ is a ground state solution of \eqref{eq3.1}.
	
	\textbf{Step 3}
	We already know $v\ge0$ and $v\not=0$. By the strong maximum principle for mixed local--nonlocal operators (see, for example, \cite{2024DipierroAMS} and references therein), we conclude that
	\[
	v>0\quad\text{in }\R^{2}.
	\]
	This completes the proof.
\end{proof}

\section{Ground state solution of the singularly perturbed problem}

\begin{Lem}\label{Lem5.1}
	Assume that $(V)$ and $(f_{1})$--$(f_{3})$ hold. Then there exists a constant
	$\alpha>0$, independent of $\varepsilon$, such that
	\[
	\|u\|_{\varepsilon} \ge \alpha,
	\qquad \forall\,u\in\mathcal{N}_{\varepsilon}.
	\]
\end{Lem}

\begin{proof}
	We use the standard convention that $f(t)=0$ for $t\le 0$, hence $F(t)=0$ for $t\le 0$.
	Combining $(f_{1})$ with $(f_{2})$, for any $\eta>0$ there exist $q>1$, $k>0$ and $C_{\eta}>0$ such that
	\[
	|f(s)|
	\le \eta |s|^{\frac{2-\mu}{2}}
	+C_{\eta}|s|^{q-1}\bigl(\mathrm{e}^{k4\pi s^{2}}-1\bigr),
	\qquad \forall s\in\R .
	\]
	Set $p=\frac{4}{4-\mu}$. By Lemma~\ref{Lem2.3} and $(f_3)$,
	\[
	\int_{\R^{2}}\Bigl(\frac{1}{|x|^{\mu}}*F(u)\Bigr)f(u)u\,dx
	\le C_0 \|F(u)\|_{p}\|f(u)u\|_{p}
	\le C_1 \|f(u)u\|_{p}^{2}.
	\]
	Using the above estimate on $f(u)u$ and $(a+b)^p\le 2^{p-1}(a^p+b^p)$, we obtain
	\begin{align}\label{eq5.1}
		\int_{\R^{2}}\Bigl(\frac{1}{|x|^{\mu}}*F(u)\Bigr)f(u)u\,dx
		&\le C_2\Biggl(
		\eta \int_{\R^{2}}|u|^{2}\,dx
		+C_{\eta}\int_{\R^{2}}
		|u|^{\frac{4q}{4-\mu}}
		\bigl(\mathrm{e}^{\frac{4k}{4-\mu}4\pi u^{2}}-1\bigr)\,dx
		\Biggr)^{\frac{4-\mu}{2}}.
	\end{align}

	We estimate the second integral in \eqref{eq5.1}. Let $A=\frac{4q}{4-\mu}>2$ and $B=\frac{4k}{4-\mu}4\pi$.
	By H\"older's inequality,
	\begin{equation}\label{eq5.2}
	  \int_{\R^{2}} |u|^{A}\bigl(\mathrm{e}^{B u^{2}}-1\bigr)\,dx
	\le \|u\|_{2A}^{A}
	\left(\int_{\R^{2}}\bigl(\mathrm{e}^{2B u^{2}}-1\bigr)\,dx\right)^{\frac12}.  
	\end{equation}
	By the Sobolev embedding on $\R^2$, $\|u\|_{2A}\le C \|u\|_{H^1(\R^2)}\le C\|u\|_{\varepsilon}$.
	Set $v=u/\|u\|_{\varepsilon}$. Then $\|\nabla v\|_2\le 1$ and, since $V(\varepsilon x)\ge V_0>0$,
	\[
	\|v\|_2^2 \le \frac{1}{V_0}\int_{\R^2}V(\varepsilon x)v^2\,dx \le \frac{1}{V_0}.
	\]
    We now distinguish two cases.

	Case 1. If
	\[
	\|u\|_{\varepsilon}^{2} < \frac{4-\mu}{8k},
	\]
	then Proposition~\ref{pro2.2} applied to $v$ (with $M=1/\sqrt{V_0}$) yields
	\[
	\int_{\R^{2}}\bigl(\mathrm{e}^{2B u^{2}}-1\bigr)\,dx
	=\int_{\R^{2}}\Bigl(\mathrm{e}^{(2B\|u\|_{\varepsilon}^{2}) v^{2}}-1\Bigr)\,dx
	\le C_3,
	\]
	where $C_3>0$ is independent of $u$ and $\varepsilon$. Consequently, under \eqref{eq5.2}
	\[
	\int_{\R^{2}} |u|^{A}\bigl(\mathrm{e}^{B u^{2}}-1\bigr)\,dx
	\le C_4 \|u\|_{\varepsilon}^{A}.
	\]
	Plugging this bound into \eqref{eq5.1} and using $\|u\|_2\le V_0^{-1/2}\|u\|_\varepsilon$, we obtain
	\begin{equation}\label{eq5.3}
		\int_{\R^{2}}\Bigl(\frac{1}{|x|^{\mu}}*F(u)\Bigr)f(u)u\,dx
		\le \eta C_5 \|u\|_{\varepsilon}^{4-\mu}
		+C_5 C_{\eta}\|u\|_{\varepsilon}^{2q},
	\end{equation}
	for some $C_5>0$ independent of $\varepsilon$.

     Since $u\in\mathcal{N}_{\varepsilon}$,
	\[
	\|u\|_{\varepsilon}^{2}
	=\int_{\R^{2}}\Bigl(\frac{1}{|x|^{\mu}}*F(u)\Bigr)f(u)u\,dx.
	\]
	Combining with \eqref{eq5.3} gives
	\[
	\|u\|_{\varepsilon}^{2}
	\le \eta C_5 \|u\|_{\varepsilon}^{4-\mu}
	+C_5 C_{\eta}\|u\|_{\varepsilon}^{2q}.
	\]
	Let $t=\|u\|_{\varepsilon}>0$. Dividing by $t^2$ yields
	\[
	1 \le \eta C_5 t^{2-\mu} + C_5 C_{\eta} t^{2q-2}.
	\]
	Since $2-\mu>0$ and $2q-2>0$, the right-hand side tends to $0$ as $t\to 0^+$.
	Hence there exists $\alpha_1\in(0,1)$, independent of $\varepsilon$, such that the above inequality cannot hold for $t\in(0,\alpha_1]$.
	Therefore $\|u\|_{\varepsilon}\ge \alpha_1$ in Case 1.

	Case 2. If
	\[
	\|u\|_{\varepsilon}^{2} \ge \frac{4-\mu}{8k},
	\]
	then
	\[
	\|u\|_{\varepsilon}\ge \sqrt{\frac{4-\mu}{8k}}=:\alpha_0.
	\]

	Finally, setting
	\[
	\alpha := \min\{\alpha_0,\alpha_1\}>0,
	\]
	we conclude that $\|u\|_{\varepsilon}\ge \alpha$ for all $u\in\mathcal{N}_{\varepsilon}$, with $\alpha$ independent of $\varepsilon$.

\end{proof}

\begin{Lem}\label{Lem5.2}
	Assume that $(V)$ and $(f_1)$--$(f_6)$ hold, and let $c_\varepsilon$ be the mountain pass level associated with $\mathcal J_\varepsilon$ (see Lemma~\ref{Lem2.8}). Then
	\[
	\lim_{\varepsilon\to0} c_\varepsilon = c_{V_0},
	\]
	where $c_{V_0}$ is the minimax value defined in \eqref{eq3.2} with $a\equiv V_0$. Hence, by Lemma~\ref{Lem3.1}, there exists $\varepsilon_0>0$ such that
	\[
	c_\varepsilon < \frac{4-\mu}{8}\bigl(1+C_s\bigr),
	\qquad \forall\,\varepsilon\in(0,\varepsilon_0).
	\]
	Moreover, one has $c_{V_0}<c_{V_\infty}$ and therefore
	\[
	\lim_{\varepsilon\to0} c_\varepsilon = c_{V_0} < c_{V_\infty}.
	\]
\end{Lem}

\begin{proof}
	Let $w\in W_{V_0}$ be a positive ground state of \eqref{eq3.1} with $a\equiv V_0$, so that
	\[
	w\in\mathcal N_{V_0},
	\qquad
	\mathcal I_{V_0}(w)=c_{V_0}.
	\]
	Fix $\delta>0$ and choose $\varphi_\delta\in C_0^\infty(\R^2)$, $\varphi_\delta\ge 0$, such that
	\[
	\|\varphi_\delta-w\|_{W_{V_0}}<\delta.
	\]
	By Lemma~\ref{Lem2.7} for $\mathcal I_{V_0}$, there exists a unique $t_\delta>0$ such that
	$t_\delta\varphi_\delta\in\mathcal N_{V_0}$. Set
	\[
	w_\delta=t_\delta\varphi_\delta\in C_0^\infty(\R^2)\cap\mathcal N_{V_0}.
	\]
	Taking $\delta$ smaller if necessary, we may assume
	\[
	\mathcal I_{V_0}(w_\delta)<c_{V_0}+\delta.
	\]

	Let $x_0\in\R^2$ be such that $V(x_0)=V_0$.
	Fix any sequence $\varepsilon_n\to 0$ and define the translated test functions
	\[
	w_n(x)=w_\delta\Bigl(x-\frac{x_0}{\varepsilon_n}\Bigr).
	\]
	Then $w_n\in W_{\varepsilon_n}$ and $\{w_n\}$ is bounded in $W_{\varepsilon_n}$.
	Moreover, by translation invariance of the local and fractional terms and of the Choquard term,
	\[
	\int_{\R^2}|\nabla w_n|^2\,dx=\int_{\R^2}|\nabla w_\delta|^2\,dx,
	\qquad
	\iint_{\R^2\times\R^2}\frac{|w_n(x)-w_n(y)|^2}{|x-y|^{2+2s}}\,dx\,dy
	=
	\iint_{\R^2\times\R^2}\frac{|w_\delta(x)-w_\delta(y)|^2}{|x-y|^{2+2s}}\,dx\,dy,
	\]
	and
	\[
	\int_{\R^2}\Bigl(\frac{1}{|x|^\mu}*F(w_n)\Bigr)F(w_n)\,dx
	=
	\int_{\R^2}\Bigl(\frac{1}{|x|^\mu}*F(w_\delta)\Bigr)F(w_\delta)\,dx.
	\]
	For the potential term, by the change of variables $z=x-\frac{x_0}{\varepsilon_n}$,
	\[
	\int_{\R^2}V(\varepsilon_n x)w_n^2\,dx
	=
	\int_{\R^2}V(x_0+\varepsilon_n z)\,w_\delta^2(z)\,dz
	\longrightarrow
	V_0\int_{\R^2}w_\delta^2\,dz,
	\]
	since $w_\delta$ has compact support and $V$ is continuous.

	For each $n$, by Lemma~\ref{Lem2.7} applied to $\mathcal J_{\varepsilon_n}$, there exists a unique $t_n>0$ such that
	\[
	t_n w_n\in\mathcal N_{\varepsilon_n}.
	\]

	We show that $\{t_n\}$ is bounded and bounded away from $0$.
	By Lemma~\ref{Lem5.1}, there exists $\alpha>0$ independent of $\varepsilon$ such that
	\[
	\|u\|_\varepsilon\ge \alpha\qquad \forall\,u\in\mathcal N_\varepsilon.
	\]
	Hence $\|t_n w_n\|_{\varepsilon_n}\ge\alpha$ for all $n$. Since $\sup_n\|w_n\|_{\varepsilon_n}<\infty$, it follows that $t_n\ge c_1>0$ for some $c_1$ independent of $n$.

	Assume by contradiction that $t_n\to+\infty$.
	Since $w_\delta\ge0$ and $w_\delta\not=0$, there exist a measurable set $E\subset\R^2$ with $|E|>0$ and a constant $m>0$ such that
	\[
	w_\delta(x)\ge m\quad \text{for a.e. }x\in E.
	\]
	 we have $w_n(x+\frac{x_0}{\varepsilon_n})=w_\delta(x)$ on $E$, hence $t_n w_n(x+\frac{x_0}{\varepsilon_n})=t_nw_\delta(x)\ge t_n m$ on $E$.

	Fix $\sigma\in(0,\beta)$. By $(f_5)$ there exists $T_\sigma>0$ such that
	\[
	t f(t)F(t)\ge (\beta-\sigma)\,e^{8\pi t^2}\qquad \text{for all }t\ge T_\sigma.
	\]
Then there exists $N$, when $n>N$, $t_nm>T_\sigma$, we have
	\[
	(t_n w_n)\,f(t_n w_n)\,F(t_n w_n)\ge (\beta-\sigma)\,e^{8\pi (t_n w_n)^2}\ge (\beta-\sigma)\,e^{8\pi t_n^2 m^2}
	.
	\]

	Since $t_n w_n\in\mathcal N_{\varepsilon_n}$,
	\[
	t_n^2\|w_n\|_{\varepsilon_n}^2
	=\int_{\R^2}\Bigl(\frac{1}{|x|^\mu}*F(t_n w_n)\Bigr) f(t_n w_n)\,t_n w_n\,dx.
	\]
	Set $F=E+\frac{x_0}{\varepsilon_n}$,\,\,$D=\sup\{|x-y|:\ x,y\in F\}<\infty$ and $|F|>0$.
    For $x\in F$,
	\[
	\Bigl(\frac{1}{|x|^\mu}*F(t_n w_n)\Bigr)(x)
	\ge \int_F \frac{F(t_n w_n(y+\frac{x_0}{\varepsilon_n}))}{|x-y|^\mu}\,dy
	\ge \frac{1}{D^\mu}\int_F F(t_n w_n(y+\frac{x_0}{\varepsilon_n}))\,dy.
	\]
	Therefore,
	\[
	t_n^2\|w_n\|_{\varepsilon_n}^2
	\ge \frac{1}{D^\mu}
	\Bigl(\int_F F(t_n w_n)\,dx\Bigr)
	\Bigl(\int_F f(t_n w_n)\,t_n w_n\,dx\Bigr).
	\]
	By Cauchy--Schwarz inequality,
	\[
	\Bigl(\int_F F(t_n w_n)\,dx\Bigr)\Bigl(\int_F f(t_n w_n)\,t_n w_n\,dx\Bigr)
	\ge \Bigl(\int_F \sqrt{F(t_n w_n)\,f(t_n w_n)\,t_n w_n}\,dx\Bigr)^2.
	\]
	On $E$ and for $n$ large,
	\[
	\sqrt{F(t_n w_n)\,f(t_n w_n)\,t_n w_n}
	\ge \sqrt{\beta-\sigma}\,e^{4\pi (t_n w_n)^2}
	\ge \sqrt{\beta-\sigma}\,e^{4\pi m^2 t_n^2}.
	\]
	Hence
	\[
	t_n^2\|w_n\|_{\varepsilon_n}^2
	\ge \frac{\beta-\sigma}{D^\mu}|F|^2 e^{8\pi m^2 t_n^2}=\frac{\beta-\sigma}{D^\mu}|E|^2 e^{8\pi m^2 t_n^2},
	\]
	which contradicts the boundedness of $\|w_n\|_{\varepsilon_n}$ and the fact that the left-hand side grows at most like $t_n^2$.
	Therefore $\{t_n\}$ is bounded above. Thus, up to a subsequence,
	\[
	t_n\to t_0>0.
	\]

	We now prove $t_0=1$.
	From $\langle \mathcal J'_{\varepsilon_n}(t_n w_n),t_n w_n\rangle=0$ we have
	\[
	t_n^2\|w_n\|_{\varepsilon_n}^2
	=
	\int_{\R^2}\Bigl(\frac{1}{|x|^\mu}*F(t_n w_n)\Bigr) f(t_n w_n)\,t_n w_n\,dx.
	\]
	Define $\tilde w_n(x)=w_\delta(x)$ and note that $t_n w_n(x)=t_n \tilde w_n\bigl(x-\frac{x_0}{\varepsilon_n}\bigr)$.
	By translation invariance of the Choquard form, the right-hand side equals
	\[
	\int_{\R^2}\Bigl(\frac{1}{|x|^\mu}*F(t_n w_\delta)\Bigr) f(t_n w_\delta)\,t_n w_\delta\,dx.
	\]
	Since $w_\delta\in C_0^\infty(\R^2)$ and $t_n\to t_0$, we have pointwise convergence
	$t_n w_\delta\to t_0 w_\delta$ and a uniform bound $|t_n w_\delta|\le C$.
	By continuity of $F$ and $f(\cdot)\cdot$, it follows that, with $p=\frac{4}{4-\mu}$,
	\[
	F(t_n w_\delta)\to F(t_0 w_\delta)\quad\text{in }L^p(\R^2),
	\qquad
	f(t_n w_\delta)\,t_n w_\delta\to f(t_0 w_\delta)\,t_0 w_\delta\quad\text{in }L^p(\R^2).
	\]
	Then Lemma~\ref{Lem2.3} yields
	\[
	\int_{\R^2}\Bigl(\frac{1}{|x|^\mu}*F(t_n w_\delta)\Bigr) f(t_n w_\delta)\,t_n w_\delta\,dx
	\longrightarrow
	\int_{\R^2}\Bigl(\frac{1}{|x|^\mu}*F(t_0 w_\delta)\Bigr) f(t_0 w_\delta)\,t_0 w_\delta\,dx.
	\]
	Moreover, from the convergence of $\|w_n\|_{\varepsilon_n}^2$ to $\|w_\delta\|_{V_0}^2$, we can pass to the limit in the Nehari identity and get
	\[
	\langle \mathcal I'_{V_0}(t_0 w_\delta),t_0 w_\delta\rangle=0,
	\]
	so $t_0 w_\delta\in\mathcal N_{V_0}$. Since $w_\delta\in\mathcal N_{V_0}$ and the Nehari scaling is unique by Lemma~\ref{Lem2.7},
	we conclude $t_0=1$, hence $t_n\to 1$.

	Using $t_n w_n\in\mathcal N_{\varepsilon_n}$ and the definition of $c_{\varepsilon_n}$,
	\[
	c_{\varepsilon_n}\le \mathcal J_{\varepsilon_n}(t_n w_n).
	\]
	By the convergences above and $t_n\to 1$, we obtain
	\[
	\mathcal J_{\varepsilon_n}(t_n w_n)
	=
	\mathcal I_{V_0}(t_n w_\delta)
	+\frac{t_n^2}{2}\int_{\R^2}\bigl(V(x_0+\varepsilon_n x)-V_0\bigr)w_\delta^2\,dx
	\longrightarrow
	\mathcal I_{V_0}(w_\delta).
	\]
	Therefore,
	\[
	\limsup_{n\to\infty}c_{\varepsilon_n}\le \mathcal I_{V_0}(w_\delta)\le c_{V_0}+\delta.
	\]
	Letting $\delta\to 0$ gives
	\[
	\limsup_{\varepsilon\to 0}c_\varepsilon\le c_{V_0}.
	\]

	On the other hand, since $V(\varepsilon x)\ge V_0$ for all $x\in\R^2$, we have
	\[
	\mathcal J_\varepsilon(u)\ge \mathcal I_{V_0}(u)\qquad \forall\,u\in W_\varepsilon,
	\]
	which implies
	\[
	c_\varepsilon\ge c_{V_0},
	\qquad
	\forall\,\varepsilon>0.
	\]
	Consequently, $\lim_{\varepsilon\to 0}c_\varepsilon=c_{V_0}$.

	By Lemma~\ref{Lem3.1} with $a\equiv V_0$, we have
	\[
	c_{V_0}<\frac{4-\mu}{8}(1+C_s),
	\]
	and hence $c_\varepsilon<\frac{4-\mu}{8}(1+C_s)$ for all $\varepsilon$ small enough.

	Finally, to compare $c_{V_0}$ and $c_{V_\infty}$, note that if $a_1<a_2$ then $c_{a_1}<c_{a_2}$.
	Indeed, for any $u\in\mathcal N_{a_2}$ one has
	\[
	\langle \mathcal I'_{a_1}(u),u\rangle
	=
	\langle \mathcal I'_{a_2}(u),u\rangle-(a_2-a_1)\|u\|_2^2
	=-(a_2-a_1)\|u\|_2^2<0,
	\]
	so there exists a unique $t(u)\in(0,1)$ such that $t(u)u\in\mathcal N_{a_1}$.
	Then $\mathcal I_{a_1}(t(u)u)<\mathcal I_{a_2}(u)$, and taking the infimum over $u\in\mathcal N_{a_2}$ yields $c_{a_1}<c_{a_2}$.
	Since $V_0<V_\infty$, it follows that $c_{V_0}<c_{V_\infty}$ and thus
	\[
	\lim_{\varepsilon\to 0}c_\varepsilon=c_{V_0}<c_{V_\infty}.
	\]
\end{proof}

\begin{Lem}\label{Lem5.3}
Assume that $(V)$ and $(f_1)$--$(f_6)$ hold. Let $\varepsilon\in(0,\varepsilon_0)$ and let $\{u_n\}$ be a $(PS)_{c_{\varepsilon}}$ sequence for $\mathcal{J}_{\varepsilon}$, that is
\[
\mathcal{J}_{\varepsilon}(u_n)\to c_{\varepsilon},
\qquad
\mathcal{J}_{\varepsilon}'(u_n)\to 0
\quad\text{in }W_\varepsilon^{-1}
\quad\text{as }n\to\infty.
\]
Then $\mathcal{J}_{\varepsilon}$ satisfies the $(PS)_{c_{\varepsilon}}$ condition: there exists $u_{\varepsilon}\in W_{\varepsilon}$ such that
\[
u_n\to u_{\varepsilon}\quad\text{strongly in }W_{\varepsilon}.
\]
\end{Lem}

\begin{proof}
By Lemma~\ref{Lem5.2} we have
\[
c_{\varepsilon}<\frac{4-\mu}{8}(1+C_s)
\quad\text{for all }\varepsilon\in(0,\varepsilon_0).
\]
Arguing as in Lemma~\ref{Lem4.2}, $\{u_n\}$ is bounded in $W_\varepsilon$.
Hence, up to a subsequence, there exists $u_\varepsilon\in W_\varepsilon$ such that
\[
u_n\rightharpoonup u_\varepsilon\quad\text{in }W_\varepsilon,
\qquad
u_n\to u_\varepsilon\quad\text{in }L^p_{\mathrm{loc}}(\R^2)\ (p\ge1),
\qquad
u_n(x)\to u_\varepsilon(x)\ \text{a.e. in }\R^2,
\]
and $\mathcal J'_\varepsilon(u_\varepsilon)=0$. Moreover, by Lemma~\ref{Lem4.2} we may assume $u_n\ge0$ for all $n$.

We claim that $u_\varepsilon\not=0$. Assume by contradiction that $u_\varepsilon=0$. Suppose that for some $r>0$,
\[
\lim_{n\to\infty}\sup_{y\in\R^{2}}
\int_{B_{r}(y)}|u_{n}|^{2}\,dx=0.
\]
Then Lions' lemma implies $u_n\to0$ in $L^p(\R^2)$ for every $p\in(2,\infty)$.
Using $(f_1)$--$(f_2)$ and the Trudinger--Moser control as in Lemma~\ref{Lem2.5}, one obtains
\[
\|F(u_n)\|_{\frac{4}{4-\mu}}\to0,
\qquad
\|f(u_n)u_n\|_{\frac{4}{4-\mu}}\to0.
\]
By Lemma~\ref{Lem2.3},
\[
\int_{\R^2}\Bigl(\frac{1}{|x|^\mu}*F(u_n)\Bigr)F(u_n)\,dx\to0,
\qquad
\int_{\R^2}\Bigl(\frac{1}{|x|^\mu}*F(u_n)\Bigr)f(u_n)u_n\,dx\to0.
\]
Since $\langle \mathcal J'_\varepsilon(u_n),u_n\rangle=o_n(1)$, it follows that
\[
\|u_n\|_\varepsilon^2
=
\int_{\R^2}\Bigl(\frac{1}{|x|^\mu}*F(u_n)\Bigr)f(u_n)u_n\,dx+o_n(1)\to0,
\]
hence $\mathcal J_\varepsilon(u_n)\to0$, contradicting $c_\varepsilon>0$. Therefore vanishing cannot occur.

Thus, by Lions' lemma, there exist $r>0$, $\delta>0$ and $\{y_n\}\subset\R^2$ such that
\[
\liminf_{n\to\infty}\int_{B_r(y_n)}|u_n|^2\,dx\ge \delta.
\]
Since $u_n\to0$ in $L^2_{\mathrm{loc}}(\R^2)$, we have $|y_n|\to\infty$.
Define $\tilde u_n(x):=u_n(x+y_n)$. Then
\[
\liminf_{n\to\infty}\int_{B_r(0)}|\tilde u_n|^2\,dx\ge \delta,
\]
and, up to a subsequence,
\[
\tilde u_n\rightharpoonup \tilde u \ \text{in }H^1(\R^2),
\qquad
\tilde u_n\to \tilde u\ \text{in }L^p_{\mathrm{loc}}(\R^2)\ (p\ge1),
\qquad
\tilde u_n(x)\to \tilde u(x)\ \text{a.e.},
\]
with $\tilde u\not=0$ and $\tilde u\ge0$.
Choose $\zeta>0$ and a measurable set $E\subset\R^2$ with $|E|>0$ such that $\tilde u\ge \zeta$ a.e. on $E$.
Choose $M>0$ so that $E_M:=\{x\in E:\ \zeta\le \tilde u(x)\le M\}$ has positive measure, and choose a bounded measurable subset $\Omega\subset E_M$ with $|\Omega|>0$.

For each $n$, by Lemma~\ref{Lem2.7} applied to $\mathcal I_{V_\infty}$, there exists a unique $t_n>0$ such that
\[
t_n u_n\in \mathcal N_{V_\infty},
\qquad
\bigl\langle \mathcal I'_{V_\infty}(t_nu_n),t_nu_n\bigr\rangle=0.
\]
Arguing as in Lemma~\ref{Lem5.2}, using $(f_3)$, one shows that $\{t_n\}$ is bounded and bounded away from $0$; hence, up to a subsequence,
\[
t_n\to t_0>0.
\]

Subtracting the Nehari identity for $\mathcal I_{V_\infty}(t_nu_n)$ and the identity $\langle \mathcal J'_\varepsilon(u_n),u_n\rangle=o_n(1)$ yields
\begin{equation}\label{eq5.4}
\int_{\R^2}(V_\infty-V(\varepsilon x))u_n^2\,dx+o_n(1)
=
\iint_{\R^2\times\R^2}
\frac{F(t_n\tilde u_n(y))f(t_n\tilde u_n(x))t_n\tilde u_n(x)-t_n^2F(\tilde u_n(y))f(\tilde u_n(x))\tilde u_n(x)}{t_n^2|x-y|^\mu}\,dx\,dy.
\end{equation}
Fix $\eta>0$. By $(V)$ there exists $R>0$ such that \begin{equation*}
V(\varepsilon x) \geq V_{\infty}-\eta, \text { for any }|x| \geq R .
\end{equation*}
Using the fact that $u_n \rightarrow 0$ in $L^2\left(B_R(0)\right)$, we conclude that
\begin{equation}\label{eq5.5}
	\begin{aligned}
&\int_{\mathbb{R}^2}\left(V_{\infty}-V(\varepsilon x)\right)  \left|u_n\right|^2 \mathrm{~d} x \\& \leq \int_{B_R(0)}\left(V_{\infty}-V_0\right) \left|u_n\right|^2 \mathrm{~d} x+\eta \int_{B_R^c(0)} \left|u_n\right|^2 \mathrm{~d} x \\
& \leq \eta C+o_n(1)
\end{aligned}
\end{equation}
where $C=\sup _{n \in \mathbb{N}}\left|u_n\right|_2^2$.

We claim that $t_0=1$. Assume first that $t_0>1$. Then $t_n>1$ for $n$ large.
For $a>0$ define $H(a)=F(a)/a$. By $(f_3)$ we have $a f(a)-F(a)\ge (\theta-1)F(a)\ge0$, hence $H'(a)\ge0$ and $H$ is nondecreasing on $(0,\infty)$.
Using also $(f_4)$, for $t>1$ and $a,b>0$ one has $F(ta)\ge tF(a)$ and $f(tb)\ge f(b)$, hence
\[
F(ta)f(tb)\,tb-t^2F(a)f(b)\,b\ge0.
\]
Therefore, for $n$ large, the integrand
\[
G_n(x,y)
=
\frac{F(t_n\tilde u_n(y))f(t_n\tilde u_n(x))t_n\tilde u_n(x)-t_n^2F(\tilde u_n(y))f(\tilde u_n(x))\tilde u_n(x)}{t_n^2|x-y|^\mu}
\]
is nonnegative on $\R^2\times\R^2$.
Moreover, for a.e. $(x,y)\in\Omega\times\Omega$, since $\tilde u(x),\tilde u(y)\in[\zeta,M]$, the continuity of $F,f$ and $t_n\to t_0>1$ imply
\[
G_n(x,y)\to
G(x,y)
:=
\frac{F(t_0\tilde u(y))f(t_0\tilde u(x))t_0\tilde u(x)-t_0^2F(\tilde u(y))f(\tilde u(x))\tilde u(x)}{t_0^2|x-y|^\mu},
\]
and $G(x,y)>0$ for a.e. $(x,y)\in\Omega\times\Omega$. Since $\mu<2$ and $\Omega$ is bounded, $G\in L^1(\Omega\times\Omega)$.
By Fatou's lemma,
\[
\liminf_{n\to\infty}\iint_{\Omega\times\Omega}G_n\,dx\,dy
\ge
\iint_{\Omega\times\Omega}G\,dx\,dy
>0.
\]
On the other hand, since $G_n\ge0$ and \eqref{eq5.4},\eqref{eq5.5} holds,
\[
0<\iint_{\Omega\times\Omega}G\,dx\,dy
\le \eta C,
\]
 since the arbitrariness of $\eta$, which is a contradiction. Thus $t_0>1$ is impossible.

Assume next that $t_0<1$. Then $t_n<1$ for $n$ large.
By $(f_4)$, for all $a,b>0$ and $t\in(0,1)$,
\[
F(ta)\le tF(a),
\qquad
f(tb)\,tb\le t f(b)\,b,
\]
hence
\[
F(ta)\bigl(f(tb)\,tb-F(tb)\bigr)
\le t^2 F(a)\bigl(f(b)\,b-F(b)\bigr).
\]
Using
\[
\mathcal I_{V_\infty}(v)-\frac12\langle \mathcal I'_{V_\infty}(v),v\rangle
=
\frac12\int_{\R^2}\Bigl(\frac1{|x|^\mu}*F(v)\Bigr)\bigl(f(v)v-F(v)\bigr)\,dx,
\]
and $\langle \mathcal I'_{V_\infty}(t_nu_n),t_nu_n\rangle=0$, we obtain for $n$ large
\[
\mathcal I_{V_\infty}(t_nu_n)
=
\frac12\int_{\R^2}\Bigl(\frac1{|x|^\mu}*F(t_nu_n)\Bigr)\bigl(f(t_nu_n)\,t_nu_n-F(t_nu_n)\bigr)\,dx
\le
\mathcal J_\varepsilon(u_n)-\frac12\langle \mathcal J'_\varepsilon(u_n),u_n\rangle+o_n(1).
\]
Letting $n\to\infty$ yields $c_{V_\infty}\le c_\varepsilon$, contradicting Lemma~\ref{Lem5.2}.
Thus $t_0<1$ is impossible, and we conclude $t_0=1$.

Using $t_n\to1$, the $C^1$ regularity of $\mathcal J_\varepsilon$ and $\mathcal J'_\varepsilon(u_n)\to0$, we have
\[
\mathcal J_\varepsilon(t_nu_n)=\mathcal J_\varepsilon(u_n)+o_n(1)=c_\varepsilon+o_n(1).
\]
Moreover,
\[
\mathcal I_{V_\infty}(t_nu_n)
=
\mathcal J_\varepsilon(t_nu_n)
+\frac{t_n^2}{2}\int_{\R^2}(V_\infty-V(\varepsilon x))u_n^2\,dx
\leq
c_\varepsilon+\eta C+o_n(1),
\]
where we used \eqref{eq5.5}. Since $\mathcal I_{V_\infty}(t_nu_n)\ge c_{V_\infty}$ and the arbitrariness of $\eta$, letting $n\to\infty$ gives $c_{V_\infty}\le c_\varepsilon$, again a contradiction.
This shows that our assumption $u_\varepsilon=0$ is false, hence $u_\varepsilon\not=0$.

Finally, we show $u_n\to u_\varepsilon$ strongly in $W_\varepsilon$.
Since $u_\varepsilon\ne0$ and $\mathcal J'_\varepsilon(u_\varepsilon)=0$, we have $u_\varepsilon\in\mathcal N_\varepsilon$ and therefore
\[
\mathcal J_\varepsilon(u_\varepsilon)\ge c_\varepsilon.
\]
On the other hand, by $(f_3)$,
\[
\mathcal J_\varepsilon(u)-\frac{1}{2\theta}\langle \mathcal J'_\varepsilon(u),u\rangle
=
\left(\frac12-\frac{1}{2\theta}\right)\|u\|_\varepsilon^2
+
\int_{\R^2}\Bigl(\frac1{|x|^\mu}*F(u)\Bigr)
\left(\frac{1}{2\theta} f(u)u-\frac12F(u)\right)\,dx,
\]
and the integral term is nonnegative. Hence, using weak lower semicontinuity and Fatou's lemma,
\[
\mathcal J_\varepsilon(u_\varepsilon)
=
\mathcal J_\varepsilon(u_\varepsilon)-\frac{1}{2\theta}\langle \mathcal J'_\varepsilon(u_\varepsilon),u_\varepsilon\rangle
\le
\liminf_{n\to\infty}
\left(\mathcal J_\varepsilon(u_n)-\frac{1}{2\theta}\langle \mathcal J'_\varepsilon(u_n),u_n\rangle\right)
=
\lim_{n\to\infty}\mathcal J_\varepsilon(u_n)
=
c_\varepsilon.
\]
Therefore $\mathcal J_\varepsilon(u_\varepsilon)=c_\varepsilon$.

Set
\[
A_n=\left(\frac12-\frac{1}{2\theta}\right)\|u_n\|_\varepsilon^2,
\qquad
B_n=\int_{\R^2}\Bigl(\frac1{|x|^\mu}*F(u_n)\Bigr)
\left(\frac{1}{2\theta} f(u_n)u_n-\frac12F(u_n)\right)\,dx,
\]
and define $A,B$ analogously with $u_n$ replaced by $u_\varepsilon$.
Then $A_n\ge0$, $B_n\ge0$, $A\ge0$, $B\ge0$, and
\[
A_n+B_n
=
\mathcal J_\varepsilon(u_n)-\frac{1}{2\theta}\langle \mathcal J'_\varepsilon(u_n),u_n\rangle
=
c_\varepsilon+o_n(1),
\qquad
A+B
=
\mathcal J_\varepsilon(u_\varepsilon)
=
c_\varepsilon.
\]
Moreover, weak lower semicontinuity gives $\liminf\limits_{n\to\infty}A_n\ge A$, and Fatou's lemma gives $\liminf\limits_{n\to\infty}B_n\ge B$.
Hence
\[
c_\varepsilon
=
\liminf_{n\to\infty}(A_n+B_n)
\ge
\liminf_{n\to\infty}A_n+\liminf_{n\to\infty}B_n
\ge
A+B
=
c_\varepsilon,
\]
so all inequalities are equalities and in particular $\lim\limits_{n\to\infty}A_n=A$. Therefore
\[
\|u_n\|_\varepsilon\to \|u_\varepsilon\|_\varepsilon.
\]
Since $u_n\rightharpoonup u_\varepsilon$ in the Hilbert space $W_\varepsilon$, this implies
\[
u_n\to u_\varepsilon \quad\text{strongly in }W_\varepsilon.
\]
\end{proof}

\begin{Cor}\label{Cor5.4}
For $\varepsilon>0$ sufficiently small, the minimax value $c_\varepsilon$ is achieved at some $u_\varepsilon\in W_\varepsilon$. Consequently, problem \eqref{eq2.1} admits a positive least energy solution $u_\varepsilon$ for all $\varepsilon>0$ small.
\end{Cor}

\begin{proof}
Fix $\varepsilon\in(0,\varepsilon_0)$. By Lemma~\ref{Lem2.8} there exists a $(PS)_{c_\varepsilon}$ sequence $\{u_n\}\subset W_\varepsilon$ for $\mathcal J_\varepsilon$ such that
\[
\mathcal J_\varepsilon(u_n)\to c_\varepsilon,
\qquad
\mathcal J_\varepsilon'(u_n)\to0\quad\text{in }W_\varepsilon^{-1}.
\]
By Lemma~\ref{Lem5.3}, up to a subsequence, there exists $u_\varepsilon\in W_\varepsilon$ such that
\[
u_n\to u_\varepsilon\quad\text{strongly in }W_\varepsilon.
\]
In particular,
\[
\mathcal J_\varepsilon(u_\varepsilon)=\lim_{n\to\infty}\mathcal J_\varepsilon(u_n)=c_\varepsilon,
\qquad
\mathcal J_\varepsilon'(u_\varepsilon)=0,
\]
so $c_\varepsilon$ is achieved by the critical point $u_\varepsilon$.

We may assume $u_n\ge0$ for all $n$. Indeed, set $u_n^-=\max\{-u_n,0\}$ and use the convention $f(t)=0$ for $t\le0$. Testing $\langle \mathcal J_\varepsilon'(u_n),\cdot\rangle$ with $u_n^-$ and arguing as in the standard sign estimate yields $\|u_n^-\|_\varepsilon\to0$. Hence $\{u_n^+\}$ is still a $(PS)_{c_\varepsilon}$ sequence and $u_n^+\to u_\varepsilon$ in $W_\varepsilon$, which implies $u_\varepsilon\ge0$ a.e. in $\R^2$.

Since $\mathcal J_\varepsilon(u_\varepsilon)=c_\varepsilon>0$, we have $u_\varepsilon\not=0$. Therefore $u_\varepsilon$ is a nontrivial nonnegative weak solution of \eqref{eq2.1}. By the strong maximum principle for mixed local--nonlocal operators (see, for example, \cite{2024DipierroAMS} and references therein), it follows that
\[
u_\varepsilon>0\quad\text{in }\R^2.
\]

Finally, let $w\in W_\varepsilon$ be any nontrivial critical point of $\mathcal J_\varepsilon$. Then $w\in\mathcal N_\varepsilon$ and, by $(f_3)$--$(f_5)$, there exists $T>1$ such that $\mathcal J_\varepsilon(Tw)<0$. Hence the path $\gamma(t)=tTw$ belongs to $\Gamma_\varepsilon$ and
\[
c_\varepsilon
\le \max_{t\in[0,1]}\mathcal J_\varepsilon(\gamma(t))
= \max_{s\in[0,T]}\mathcal J_\varepsilon(sw)
= \mathcal J_\varepsilon(w),
\]
where the last equality follows from $\langle \mathcal J_\varepsilon'(w),w\rangle=0$. Therefore $c_\varepsilon$ is the least energy among all nontrivial critical points, and $u_\varepsilon$ is a positive least energy solution.
\end{proof}

\section{Concentration phenomena}
\begin{Lem}\label{Lem6.1}
	Suppose that $(f_1)$ and $(f_2)$ hold. If $h\in H^1(\R^2)$, then
	\[
	\Bigl(\frac{1}{|x|^\mu}*F(h)\Bigr)\in L^\infty(\R^2).
	\]
\end{Lem}

\begin{proof}
	We split the proof into two parts: (i) $F(h)\in L^1(\R^2)\cap L^p(\R^2)$ for some
	$p>\frac{2}{2-\mu}$; (ii) a convolution estimate giving $L^\infty$.

	Step 1.	Fix $\eta>0$. By $(f_1)$ and $(f_2)$ there exist $q>1$, $k>1$ and $C_\eta>0$ such that
	\[
	|f(t)|
	\le \eta |t|^{\frac{2-\mu}{2}} + C_\eta |t|^{q-1}\bigl(e^{k4\pi t^2}-1\bigr),
	\qquad \forall\,t\in\R .
	\]
	Integrating on $[0,t]$ and using $F(t)=\int_0^t f(\tau)\,d\tau$, we obtain
	\[
	|F(t)|
	\le C\eta |t|^{\frac{4-\mu}{2}} + C C_\eta |t|^{q}\bigl(e^{k4\pi t^2}-1\bigr),
	\qquad \forall\,t\in\R ,
	\]
	for some constant $C>0$ independent of $t$. Hence, for any $p\ge1$,
	\[
	|F(h)|^p
	\le C\Bigl(|h|^{\frac{p(4-\mu)}{2}} + |h|^{pq}\bigl(e^{k4\pi h^2}-1\bigr)^p\Bigr).
	\]
	Using the elementary inequality $s^m\le C_{m,\sigma}\,(e^{\sigma s^2}-1)$ for $s\ge0$,
	with $\sigma>0$ small, we can absorb the polynomial factor into an exponential.
	Therefore there exist $\alpha_p>0$ and $C_p>0$ such that
	\[
	|F(h(x))|^p \le C_p\bigl(e^{\alpha_p h(x)^2}-1\bigr)\quad\text{a.e. in }\R^2.
	\]
	By Proposition~\ref{pro2.2}, $\int_{\R^2}(e^{\alpha_p h^2}-1)\,dx<\infty$ for every
	$\alpha_p>0$ because $h\in H^1(\R^2)$. Hence $F(h)\in L^p(\R^2)$ for every $p\ge1$.
	In particular, $F(h)\in L^1(\R^2)$ and we may choose $p>\frac{2}{2-\mu}$.

	Step 2. Fix $p>\frac{2}{2-\mu}$ and let $p'=\frac{p}{p-1}$. For any $x\in\R^2$, split
	\[
	\Bigl|\Bigl(\frac{1}{|x|^\mu}*F(h)\Bigr)(x)\Bigr|
	\le \int_{|x-y|\le1}\frac{|F(h(y))|}{|x-y|^\mu}\,dy
	+\int_{|x-y|>1}\frac{|F(h(y))|}{|x-y|^\mu}\,dy
	=: I_1(x)+I_2(x).
	\]
	For $I_2$, since $|x-y|^{-\mu}\le1$ on $\{|x-y|>1\}$, we have
	\[
	I_2(x)\le \int_{\R^2}|F(h(y))|\,dy=\|F(h)\|_{L^1(\R^2)}.
	\]
	For $I_1$, by H\"older's inequality,
	\[
	I_1(x)\le \|F(h)\|_{L^p(\R^2)}\,
	\Bigl(\int_{|z|\le1}|z|^{-\mu p'}\,dz\Bigr)^{\frac{1}{p'}}.
	\]
	The integral $\int_{|z|\le1}|z|^{-\mu p'}\,dz$ is finite provided $\mu p'<2$,
	that is,
	\[
	p' < \frac{2}{\mu}\quad\Longleftrightarrow\quad p>\frac{2}{2-\mu},
	\]
	which is exactly our choice of $p$. Hence $I_1(x)\le C\,\|F(h)\|_{L^p(\R^2)}$
	with a constant $C$ independent of $x$.
	Combining the estimates for $I_1$ and $I_2$ yields
	\[
	\sup_{x\in\R^2}\Bigl|\Bigl(\frac{1}{|x|^\mu}*F(h)\Bigr)(x)\Bigr|
	\le C\|F(h)\|_{L^p(\R^2)}+\|F(h)\|_{L^1(\R^2)}<\infty.
	\]
	Therefore $\frac{1}{|x|^\mu}*F(h)\in L^\infty(\R^2)$.
\end{proof}

\begin{Lem}\label{Lem6.2}
	Let $\varepsilon_n\to0$ and let $\{u_n\}\subset\mathcal{N}_{\varepsilon_n}$ satisfy
	\[
	\lim_{n\to\infty}\mathcal{J}_{\varepsilon_n}(u_n)=c_{V_0}.
	\]
	Then there exists a sequence $\{\tilde y_n\}\subset\R^2$ such that the translated sequence
	\[
	\tilde u_n(x)=u_n(x+\tilde y_n)
	\]
	has a convergent subsequence in $W_{V_0}$. Moreover, up to a subsequence,
	\[
	y_n=\varepsilon_n\tilde y_n\to y\in M.
	\]
\end{Lem}

\begin{proof}
	Since $u_n\in\mathcal N_{\varepsilon_n}$, we have
	\[
	\langle \mathcal{J}'_{\varepsilon_n}(u_n),u_n\rangle=0.
	\]
	Together with $\mathcal{J}_{\varepsilon_n}(u_n)\to c_{V_0}$, arguing as in Lemma~\ref{Lem4.2}
	we deduce that $\{u_n\}$ is bounded in $W_{\varepsilon_n}$. In particular, $\{u_n\}$ is
	bounded in $H^1(\R^2)$. By $(V)$, the norms $\|\cdot\|_{\varepsilon_n}$ and $\|\cdot\|_{V_0}$
	are equivalent uniformly in $n$, hence $\{u_n\}$ is bounded in $W_{V_0}$ as well.

	We claim that $\{u_n\}$ does not vanish. If it vanished, then $u_n\to0$ in $L^p(\R^2)$ for
	every $p>2$, and using $(f_1)$, $(f_2)$ and Lemma~\ref{Lem2.3} one would obtain
	$\mathcal{J}_{\varepsilon_n}(u_n)\to0$, contradicting $c_{V_0}>0$.
	Therefore there exist $r>0$, $\delta>0$ and a sequence $\{\tilde y_n\}\subset\R^2$ such that
	\[
	\liminf_{n\to\infty}\int_{B_r(\tilde y_n)}|u_n|^2\,dx\ge\delta.
	\]
	Define $\tilde u_n(x)=u_n(x+\tilde y_n)$. Then $\{\tilde u_n\}$ is bounded in $W_{V_0}$ and
	\[
	\liminf_{n\to\infty}\int_{B_r(0)}|\tilde u_n|^2\,dx\ge\delta.
	\]
	Passing to a subsequence,
	\[
	\tilde u_n\rightharpoonup \tilde u \quad \text{in } W_{V_0},
	\qquad
	\tilde u_n\to \tilde u \quad \text{in } L^p_{\mathrm{loc}}(\R^2),\ p\ge1,
	\qquad
	\tilde u_n(x)\to \tilde u(x) \ \text{a.e. in }\R^2,
	\]
	and the lower bound on $B_r(0)$ gives $\tilde u\not\equiv0$.

	Set $y_n=\varepsilon_n\tilde y_n$. Introduce the translated functional
	\[
	\widetilde{\mathcal J}_n(v)
	=\frac12\int_{\R^2}|\nabla v|^2\,dx
	+\frac14\iint_{\R^2\times\R^2}\frac{|v(x)-v(y)|^2}{|x-y|^{2+2s}}\,dx\,dy
	+\frac12\int_{\R^2}V(\varepsilon_n x+y_n)v^2\,dx
	-\frac12\int_{\R^2}\Bigl(\frac1{|x|^\mu}*F(v)\Bigr)F(v)\,dx.
	\]
	A change of variables shows
	\[
	\widetilde{\mathcal J}_n(\tilde u_n)=\mathcal{J}_{\varepsilon_n}(u_n),
	\qquad
	\langle \widetilde{\mathcal J}_n'(\tilde u_n),\tilde u_n\rangle=0.
	\]

	Let $t_n>0$ be the unique number such that $w_n=t_n\tilde u_n\in\mathcal N_{V_0}$.
	Then
	\[
	c_{V_0}\le \mathcal I_{V_0}(w_n),
	\qquad
	\mathcal I_{V_0}(w_n)\le \widetilde{\mathcal J}_n(w_n).
	\]
	Moreover, since $\tilde u_n$ lies on the Nehari manifold of $\widetilde{\mathcal J}_n$,
	Lemma~\ref{Lem2.7} yields that $t\mapsto \widetilde{\mathcal J}_n(t\tilde u_n)$ attains its
	maximum at $t=1$, hence
	\[
	\widetilde{\mathcal J}_n(w_n)=\widetilde{\mathcal J}_n(t_n\tilde u_n)
	\le \widetilde{\mathcal J}_n(\tilde u_n)
	=\mathcal{J}_{\varepsilon_n}(u_n)
	=c_{V_0}+o(1).
	\]
	Consequently,
	\[
	c_{V_0}\le \mathcal I_{V_0}(w_n)\le c_{V_0}+o(1),
	\qquad
	\mathcal I_{V_0}(w_n)\to c_{V_0}.
	\]

	There exists $\alpha>0$ such that $\|u\|_{V_0}\ge\alpha$ for all $u\in\mathcal N_{V_0}$, hence
	$\|w_n\|_{V_0}\ge\alpha$ and $\{w_n\}$ is bounded in $W_{V_0}$.
Since $\mathcal I_{V_0}(w_n)\to c_{V_0}=\inf_{\mathcal N_{V_0}}\mathcal I_{V_0}$ and $\{w_n\}\subset\mathcal N_{V_0}$,
by Ekeland's variational principle applied to $\mathcal I_{V_0}$ restricted to $\mathcal N_{V_0}$,
there exists $\{v_n\}\subset\mathcal N_{V_0}$ such that
\[
\mathcal I_{V_0}(v_n)\to c_{V_0},\qquad
\|v_n-w_n\|_{V_0}\to0,\qquad
\|(\mathcal I_{V_0}|_{\mathcal N_{V_0}})'(v_n)\|_{(W_{V_0})^{-1}}\to0.
\]
Replacing $w_n$ by $v_n$ (still denoted by $w_n$), we may assume that
\[
\mathcal I_{V_0}(w_n)\to c_{V_0},
\qquad
\|(\mathcal I_{V_0}|_{\mathcal N_{V_0}})'(w_n)\|_{(W_{V_0})^{-1}}\to0.
\]
Let $G(u)=\langle \mathcal I_{V_0}'(u),u\rangle$. Then $\mathcal N_{V_0}=\{u\neq0:\,G(u)=0\}$.
By the Lagrange multiplier rule, there exists $\lambda_n\in\mathbb{R}$ such that
\[
\mathcal I_{V_0}'(w_n)=\lambda_n G'(w_n)+o(1)\quad\text{in }(W_{V_0})^{-1}.
\]
Testing by $w_n$ and using $G(w_n)=0$ we get
\[
0=\langle \mathcal I_{V_0}'(w_n),w_n\rangle
=\lambda_n\langle G'(w_n),w_n\rangle+o(1).
\]
Moreover, since $w_n\in\mathcal N_{V_0}$, the map $t\mapsto \mathcal I_{V_0}(t w_n)$ attains its unique maximum at $t=1$,
hence $\langle G'(w_n),w_n\rangle=h_{w_n}''(1)<0$.
Therefore $\lambda_n\to0$ and consequently
\[
\mathcal I_{V_0}'(w_n)\to0\quad\text{in }(W_{V_0})^{-1}.
\]

In addition, $t_n$ is bounded and bounded away from $0$.
Indeed, boundedness follows from $w_n=t_n\tilde u_n$ and the boundedness of $\{w_n\},\{\tilde u_n\}$,
while if $t_n\to0$ then $w_n\to0$ in $W_{V_0}$ and $\mathcal I_{V_0}(w_n)\to0$, contradicting $\mathcal I_{V_0}(w_n)\to c_{V_0}>0$.
Hence there exists $c_0>0$ such that $t_n\ge c_0$ for all $n$, and thus
\[
\int_{B_r(0)}|w_n|^2\,dx=t_n^2\int_{B_r(0)}|\tilde u_n|^2\,dx\ge c_0^2\delta.
\]
	Using the same compactness argument as in the proof of Theorem~\ref{Thm4.3}, we obtain, up to a subsequence,
	\[
	w_n\to w \quad \text{strongly in } W_{V_0},
	\]
	for some $w\in W_{V_0}$ with $w\not\equiv0$. Since $t_n$ is bounded and bounded away from $0$,
	we may assume $t_n\to t_0>0$, and therefore
	\[
	\tilde u_n=\frac{w_n}{t_n}\to \frac{w}{t_0} \quad \text{strongly in } W_{V_0}.
	\]
	This proves that $\{\tilde u_n\}$ has a convergent subsequence in $W_{V_0}$.

	It remains to show that $\{y_n\}$ is bounded and its limit lies in $M$.
	Assume by contradiction that $|y_n|\to\infty$. Fix $\eta>0$. By $(V)$ there exists $R>0$
	such that $V(z)\ge V_\infty-\eta$ for all $|z|\ge R$.
	Choose $R_0>0$ so large that $\int_{B_{R_0}^c(0)}w^2\,dx\le \eta$.
	For $n$ large, $|y_n|\ge 2R$ and $\varepsilon_n R_0\le R$, hence $|\varepsilon_n x+y_n|\ge R$
	for all $|x|\le R_0$. Therefore
	\[
	\int_{\R^2}V(\varepsilon_n x+y_n)w_n^2\,dx
	\ge (V_\infty-\eta)\int_{B_{R_0}(0)}w_n^2\,dx + V_0\int_{B_{R_0}^c(0)}w_n^2\,dx.
	\]
	Letting $n\to\infty$ and using $w_n\to w$ in $L^2(\R^2)$ we obtain
	\[
	\liminf_{n\to\infty}\int_{\R^2}V(\varepsilon_n x+y_n)w_n^2\,dx
	\ge V_\infty\int_{\R^2}w^2\,dx - C\eta,
	\]
	for a constant $C$ independent of $\eta$. Since $\eta$ is arbitrary,
	\[
	\liminf_{n\to\infty}\int_{\R^2}V(\varepsilon_n x+y_n)w_n^2\,dx
	\ge V_\infty\int_{\R^2}w^2\,dx.
	\]
	Consequently,
	\[
	\liminf_{n\to\infty}\widetilde{\mathcal J}_n(w_n)
	\ge \mathcal I_{V_\infty}(w)
	=\mathcal I_{V_0}(w) + \frac12(V_\infty-V_0)\int_{\R^2}w^2\,dx
	> c_{V_0}.
	\]
	On the other hand,
	\[
	\widetilde{\mathcal J}_n(w_n)\le \widetilde{\mathcal J}_n(\tilde u_n)
	=\mathcal{J}_{\varepsilon_n}(u_n)\to c_{V_0},
	\]
	a contradiction. Hence $\{y_n\}$ is bounded.

	Up to a subsequence, $y_n\to y\in\R^2$. If $y\notin M$, then $V(y)>V_0$.
	Since $y_n\to y$ and $\varepsilon_n\to0$, we have $V(\varepsilon_n x+y_n)\to V(y)$ uniformly
	on $B_{R_0}(0)$ for every fixed $R_0>0$. Using again $w_n\to w$ in $L^2(\R^2)$, we obtain
	\[
	\lim_{n\to\infty}\int_{\R^2}V(\varepsilon_n x+y_n)w_n^2\,dx
	= V(y)\int_{\R^2}w^2\,dx,
	\]
	hence
	\[
	\lim_{n\to\infty}\widetilde{\mathcal J}_n(w_n)
	= \mathcal I_{V(y)}(w)
	=\mathcal I_{V_0}(w) + \frac12(V(y)-V_0)\int_{\R^2}w^2\,dx
	> c_{V_0},
	\]
	which contradicts $\widetilde{\mathcal J}_n(w_n)\le \mathcal{J}_{\varepsilon_n}(u_n)\to c_{V_0}$.
	Therefore $V(y)=V_0$, namely $y\in M$.
\end{proof}

Let $\varepsilon_n \to 0$ as $n\to+\infty$ and let $v_n\in W_{\varepsilon_n}$ be the
positive ground state solution of
\[
-\Delta u+(-\Delta)^s u+V(\varepsilon_n x)u
=\Bigl(\frac{1}{|x|^\mu}*F(u)\Bigr)f(u)
\quad\text{in }\R^2,
\]
given by Corollary~\ref{Cor5.4}. Then
\[
\mathcal{J}_{\varepsilon_n}(v_n)=c_{\varepsilon_n}
\quad\text{and}\quad
\langle\mathcal{J}_{\varepsilon_n}'(v_n),v_n\rangle=0,
\]
that is, $v_n\in\mathcal{N}_{\varepsilon_n}$ for every $n$.
By Lemma~\ref{Lem5.2} we know that
\[
\mathcal{J}_{\varepsilon_n}(v_n)=c_{\varepsilon_n}\longrightarrow c_{V_0}
\quad\text{as }n\to+\infty.
\]

Hence we can apply Lemma~\ref{Lem6.2} with $u_n=v_n$ and obtain a sequence
$\{\tilde y_n\}\subset\R^2$ such that
\[
\tilde v_n(x)=v_n(x+\tilde y_n)
\]
solves
\[
-\Delta u+(-\Delta)^s u+V_n(x)u
=\Bigl(\frac{1}{|x|^\mu}*F(u)\Bigr)f(u)
\quad\text{in }\R^2,
\]
where
\[
V_n(x):=V(\varepsilon_n x+\varepsilon_n\tilde y_n),
\]
and such that, up to a subsequence,
\[
\tilde v_n\to\tilde v\quad\text{in }W_{V_0},\qquad
y_n:=\varepsilon_n\tilde y_n\to y\in M .
\]

Since $\tilde v_n\to\tilde v$ in $W_{V_0}$, in particular $\tilde v_n\to\tilde v$ in $H^1(\R^2)$.
We may extract a subsequence such that
\[
\|\tilde v_n-\tilde v\|_{H^1(\R^2)}\le 2^{-n}
\quad\text{for all }n\in\N.
\]
Define
\[
h(x)=|\tilde v(x)|+\sum_{n=1}^\infty |\tilde v_n(x)-\tilde v(x)|.
\]
Then $h\in H^1(\R^2)$. Moreover, for every $n\in\N$,
\begin{equation}\label{eq6.1}
|\tilde v_n(x)|\le h(x)
\quad\text{for a.e. }x\in\R^2.
\end{equation}

\begin{Lem}\label{Lem6.3}
	Assume that {\rm(V)} and $(f_1)$–$(f_6)$ hold. Then there exists $C>0$ such that
	\[
	\|\tilde v_n\|_{L^\infty(\R^2)} \le C \quad\text{for all }n\in\N^+.
	\]
	Furthermore,
	\[
	\lim_{|x|\to+\infty}\tilde v_n(x)=0
	\quad\text{uniformly in }n\in\N^+.
	\]
\end{Lem}

\begin{proof}
Set
\[
W_n(x)=\Bigl(\frac{1}{|x|^\mu}*F(\tilde v_n)\Bigr)(x).
\]
Since each $\tilde v_n\ge0$ and $F$ is increasing on $[0,\infty)$, by \eqref{eq6.1} we have
$0\le F(\tilde v_n)\le F(h)$ a.e. in $\R^2$. Hence
\[
0\le W_n(x)\le W(x),
\qquad
W(x):=\Bigl(\frac{1}{|x|^\mu}*F(h)\Bigr)(x).
\]
By Lemma~\ref{Lem6.1}, $W\in L^\infty(\R^2)$, therefore $\{W_n\}$ is bounded in
$L^\infty(\R^2)$.

Following \cite{1960MoserCPAM}, fix $R>0$ and $0<r\le \min\{1,R/2\}$, and take
$\eta\in C^\infty(\R^2)$ such that
\[
\eta(x)=0 \ \text{if }|x|\le R-r,
\qquad
\eta(x)=1 \ \text{if }|x|\ge R,
\qquad
|\nabla\eta|\le \frac{2}{r}.
\]
For $l>0$, set
\[
\tilde v_{n,l}(x)=
\begin{cases}
\tilde v_n(x), & \tilde v_n(x)\le l,\\
l, & \tilde v_n(x)\ge l,
\end{cases}
\]
and for $\gamma>1$ define
\[
z_{n,l}(x)=\eta(x)^2\,\tilde v_{n,l}(x)^{2(\gamma-1)}\,\tilde v_n(x),
\qquad
w_{n,l}(x)=\eta(x)\,\tilde v_{n,l}(x)^{\gamma-1}\,\tilde v_n(x).
\]

We use the standard truncation inequality in the Moser iteration scheme
(see, e.g., \cite{1960MoserCPAM,2009AlvesJDF}): there exists $C>0$ independent of
$n,l,\gamma$ such that for all $x,y\in\R^2$,
\begin{equation}\label{eq6.2}
\begin{aligned}
&\frac{1}{\gamma^{2}}\bigl(w_{n,l}(x)-w_{n,l}(y)\bigr)^{2}\\
&\le
\bigl(\tilde v_n(x)-\tilde v_n(y)\bigr)\bigl(z_{n,l}(x)-z_{n,l}(y)\bigr)
+
C\,\bigl(\eta(x)-\eta(y)\bigr)^{2}
\bigl(\tilde v_n(x)^{2}\tilde v_{n,l}(x)^{2(\gamma-1)}
+\tilde v_n(y)^{2}\tilde v_{n,l}(y)^{2(\gamma-1)}\bigr).    
\end{aligned}
\end{equation}

Taking $z_{n,l}$ as a test function in the equation satisfied by $\tilde v_n$, we obtain
\begin{equation}\label{eq6.3}
\begin{aligned}
&\int_{\R^2}\nabla\tilde v_n\nabla z_{n,l}\,dx
+\frac{1}{2}\int_{\R^2}\int_{\R^2}
\frac{\bigl(\tilde v_n(x)-\tilde v_n(y)\bigr)\bigl(z_{n,l}(x)-z_{n,l}(y)\bigr)}
{|x-y|^{2+2s}}\,dx\,dy \\
&\qquad
+\int_{\R^2} V_n(x)\,\tilde v_n z_{n,l}\,dx
=\int_{\R^2} W_n(x)\,f(\tilde v_n)\,z_{n,l}\,dx.
\end{aligned}
\end{equation}

By \eqref{eq6.2}, dividing by $|x-y|^{2+2s}$ and integrating, we get
\[
\frac{1}{2}\int_{\R^2}\int_{\R^2}
\frac{\bigl(\tilde v_n(x)-\tilde v_n(y)\bigr)\bigl(z_{n,l}(x)-z_{n,l}(y)\bigr)}
{|x-y|^{2+2s}}\,dx\,dy
\ge
\frac{1}{\gamma^{2}}[w_{n,l}]_{s}^{2}
- C\,\mathcal E_{n,l},
\]
where
\[
\mathcal E_{n,l}=\frac{1}{2}
\int_{\R^2}\int_{\R^2}
\frac{\bigl(\eta(x)-\eta(y)\bigr)^{2}}{|x-y|^{2+2s}}
\bigl(\tilde v_n(x)^{2}\tilde v_{n,l}(x)^{2(\gamma-1)}
+\tilde v_n(y)^{2}\tilde v_{n,l}(y)^{2(\gamma-1)}\bigr)\,dx\,dy.
\]
Using $|\eta(x)-\eta(y)|\le \|\nabla\eta\|_\infty|x-y|$ and $r\le1$, a standard estimate yields
\[
\mathcal E_{n,l}
\le C\,\|\nabla\eta\|_\infty^{2}
\int_{\R^2}\tilde v_n^{2}\tilde v_{n,l}^{2(\gamma-1)}\,dx
\le C\int_{\R^2}\tilde v_n^{2}\tilde v_{n,l}^{2(\gamma-1)}|\nabla\eta|^{2}\,dx,
\]
with $C$ independent of $n,l,\gamma$.

Using also $V_n\ge V_0$, from \eqref{eq6.3} we infer
\begin{equation}\label{eq6.4}
\int_{\R^2}\nabla\tilde v_n\nabla z_{n,l}\,dx
+\frac{1}{\gamma^{2}}[w_{n,l}]_s^2
+V_0\int_{\R^2} \tilde v_n z_{n,l}\,dx
\le \int_{\R^2} W_n(x)\,f(\tilde v_n)\,z_{n,l}\,dx
+C\int_{\R^2}\tilde v_n^{2}\tilde v_{n,l}^{2(\gamma-1)}|\nabla\eta|^{2}\,dx.
\end{equation}

Expanding $\nabla z_{n,l}$, discarding the nonpositive truncation contribution, and applying Young's inequality, we obtain
\begin{equation}\label{eq6.5}
\begin{aligned}
&\int_{\R^2}\eta^2\tilde v_{n,l}^{2(\gamma-1)}|\nabla\tilde v_n|^2\,dx
+\frac{1}{\gamma^{2}}[w_{n,l}]_s^2
+V_0\int_{\R^2}\eta^2\tilde v_{n,l}^{2(\gamma-1)}\tilde v_n^2\,dx\\
&\quad\le
C\int_{\R^2}\eta^2\tilde v_{n,l}^{2(\gamma-1)}\tilde v_n\,|f(\tilde v_n)|
\,\tilde v_n\,dx
+C\int_{\R^2}\tilde v_n^2\tilde v_{n,l}^{2(\gamma-1)}|\nabla\eta|^2\,dx .
\end{aligned}
\end{equation}

We now estimate the nonlinear term. Since $W_n\le \|W\|_\infty$, arguing as in Lemma~\ref{Lem5.1} from $(f_1)$--$(f_2)$, for any $\varepsilon>0$ there exist
$q>2$, $k>1$ and $C_\varepsilon>0$ such that
\[
|f(t)t|
\le \varepsilon t^2 + C_\varepsilon |t|^q\bigl(e^{k4\pi t^2}-1\bigr)
\quad\text{for all }t\in\R.
\]
Therefore,
\begin{equation}\label{eq6.6}
\begin{aligned}
|W_n(x) f(\tilde v_n)\tilde v_n|
&\le \|W\|_\infty |f(\tilde v_n)\tilde v_n| \\
&\le \varepsilon \tilde v_n^2
+ C_\varepsilon |\tilde v_n|^q\bigl(e^{k4\pi \tilde v_n^2}-1\bigr)
\le \varepsilon \tilde v_n^2
+ C_\varepsilon |\tilde v_n|^q\bigl(e^{k4\pi h^2}-1\bigr).
\end{aligned}
\end{equation}
Substituting \eqref{eq6.6} into \eqref{eq6.5}, choosing $\varepsilon>0$ small, and absorbing the $\tilde v_n^2$ term into the left-hand side, we arrive at
\begin{equation}\label{eq6.7}
\begin{aligned}
&\int_{\R^2}\eta^2\tilde v_{n,l}^{2(\gamma-1)}|\nabla\tilde v_n|^2\,dx
+\frac{1}{\gamma^{2}}[w_{n,l}]_s^2
+V_0\int_{\R^2}\eta^2\tilde v_{n,l}^{2(\gamma-1)}\tilde v_n^2\,dx\\
&\quad\le
C\int_{\R^2}\eta^2\tilde v_{n,l}^{2(\gamma-1)}|\tilde v_n|^q
\bigl(e^{k4\pi h^2}-1\bigr)\,dx
+C\int_{\R^2}\tilde v_n^2\tilde v_{n,l}^{2(\gamma-1)}|\nabla\eta|^2\,dx.
\end{aligned}
\end{equation}

Using the Sobolev embedding $H^s(\R^2)\subset L^p(\R^2)$ for $p\in[2,2_s^*]$, we obtain
\begin{equation}\label{eq6.8}
\frac{1}{\gamma^{2}}[w_{n,l}]_s^2
+V_0\int_{\R^2} w_{n,l}^2\,dx
\ge \frac{C}{\gamma^2}\,\|w_{n,l}\|_{L^p(\R^2)}^2,
\end{equation}
for some $C>0$ independent of $n,l,\gamma$.

By the elementary inequality $(e^{A}-1)^m\le C_m(e^{mA}-1)$ for $A\ge0$ and all $m>1$,
Proposition~\ref{pro2.2} yields that for all $m>1$,
\[
\int_{\R^2}\bigl(e^{k4\pi h^2}-1\bigr)^m\,dx<\infty.
\]
Fix such an $m$ and set $t=\sqrt m>1$. Choose $q>\frac{2t}{t-1}$ and set
\[
\gamma=\frac{q(t-1)}{2t}>1.
\]
Then adapting the iteration argument in \cite{2009AlvesJDF} to \eqref{eq6.7},
we obtain, after letting $l\to\infty$,
\begin{equation}\label{eq6.9}
\|\tilde v_n\|_{L^\infty(|x|\ge R)}
\le C\,\|\tilde v_n\|_{L^p(|x|\ge R/2)}.
\end{equation}
A local version with cutoffs centered at $x_0\in B_R(0)$ yields
\begin{equation}\label{eq6.10}
\|\tilde v_n\|_{L^\infty(|x-x_0|\le \rho')}
\le C\,\|\tilde v_n\|_{L^p(|x-x_0|\le 2\rho')}.
\end{equation}
Using \eqref{eq6.9}, \eqref{eq6.10} and a covering argument, we deduce
\[
\|\tilde v_n\|_{L^\infty(\R^2)}\le C
\quad\text{for all }n\in\N^+,
\]
with $C$ independent of $n$.

Finally, since $\tilde v_n\to\tilde v$ in $W_{V_0}$, we have $\tilde v_n\to\tilde v$ in $L^p(\R^2)$ for some $p>2$.
Fix $\delta>0$. Choose $R>0$ such that
\[
\int_{|x|\ge R/2}|\tilde v(x)|^p\,dx<\delta,
\]
and also
\[
\int_{|x|\ge R/2}|\tilde v_n(x)|^p\,dx<\delta
\quad\text{for }n=1,2,\dots,N,
\]
for some fixed $N$.
Moreover, for all $n\ge N$,
\[
\int_{\R^2}|\tilde v_n-\tilde v|^p\,dx<\delta,
\]
hence
\[
\int_{|x|\ge R/2}|\tilde v_n(x)|^p\,dx\le C\delta
\quad\text{for all }n\in\N^+.
\]
Using \eqref{eq6.9}, we get
\[
\sup_{n\in\N^+}\|\tilde v_n\|_{L^\infty(|x|\ge R)}
\le C\,\sup_{n\in\N^+}\|\tilde v_n\|_{L^p(|x|\ge R/2)}
\le C'\delta.
\]
Since $\delta>0$ is arbitrary,
\[
\lim_{|x|\to+\infty}\tilde v_n(x)=0
\quad\text{uniformly in }n\in\N^+,
\]
and the proof is complete.
\end{proof}

\begin{Lem}\label{Lem6.4}
	There exists $\delta_0>0$ such that
	\[
	\|\tilde v_n\|_{L^\infty(\R^2)} \ge \delta_0
	\quad\text{for all }n\in\N^+.
	\]
\end{Lem}

\begin{proof}
Recalling that
\[
0<\delta \le \int_{B_r(\tilde y_n)} |v_n|^2\,dx
\]
for some $r,\delta>0$, by the change of variables $x\mapsto x+\tilde y_n$ we get
\[
0<\delta \le \int_{B_r(0)} |\tilde v_n|^2\,dx
\le |B_r|\,\|\tilde v_n\|_{L^\infty(\R^2)}^2.
\]
Set
\[
\delta_0=\sqrt{\frac{\delta}{|B_r|}}>0.
\]
Then
\[
\|\tilde v_n\|_{L^\infty(\R^2)} \ge \delta_0
\]
for all $n\in\N^+$.
\end{proof}

\noindent\textbf{Concentration of the maximum points.}
By standard regularity for the equation satisfied by $\tilde v_n$, we have
$\tilde v_n\in C(\R^2)$ for every $n$. Since $\tilde v_n(x)\to0$ as $|x|\to\infty$
uniformly in $n$ (Lemma~\ref{Lem6.3}), each $\tilde v_n$ attains its maximum in
$\R^2$. Let $b_n\in\R^2$ be such that
\[
\tilde v_n(b_n)=\|\tilde v_n\|_{L^\infty(\R^2)}.
\]
By Lemma~\ref{Lem6.4} there exists $\delta_0>0$ such that
\[
\|\tilde v_n\|_{L^\infty(\R^2)}\ge\delta_0
\quad\text{for all }n\in\N^+.
\]
Using the uniform decay in Lemma~\ref{Lem6.3}, we can fix $R>0$ such that
\[
\sup_{n\in\N^+}\sup_{|x|\ge R}|\tilde v_n(x)|<\frac{\delta_0}{2}.
\]
Hence $b_n\in B_R(0)$ for every $n$, and the sequence $(b_n)$ is bounded in $\R^2$.

Recall that $v_n$ is the ground state solution and
\[
\tilde v_n(x)=v_n(x+\tilde y_n).
\]
Therefore the global maximum of $v_n$ is attained at
\[
z_n=b_n+\tilde y_n.
\]
Moreover,
\[
\varepsilon_n z_n
=\varepsilon_n b_n+\varepsilon_n\tilde y_n
=\varepsilon_n b_n+y_n.
\]
Since $(b_n)$ is bounded, $\varepsilon_n b_n\to0$. By Lemma~\ref{Lem6.2},
$y_n=\varepsilon_n\tilde y_n\to y\in M$. Hence
\[
\lim_{n\to\infty}\varepsilon_n z_n=y\in M.
\]
Since $V$ is continuous, then
\[
\lim_{n\to\infty}V(\varepsilon_n z_n)=V(y)=V_0.
\]

If $u_\varepsilon$ is a positive solution of problem \eqref{eq2.1}, then the rescaled function
\[
w_\varepsilon(x)=u_\varepsilon\!\left(\frac{x}{\varepsilon}\right)
\]
is a positive solution of \eqref{eq1.1}. Denote by $z_\varepsilon$ and $\eta_\varepsilon$
the global maximum points of $u_\varepsilon$ and $w_\varepsilon$, respectively. The change of variables gives
\[
\eta_\varepsilon=\varepsilon z_\varepsilon.
\]
Consequently, 
\[
\lim_{\varepsilon\to0}V(\eta_\varepsilon)
=\lim_{n\to\infty}V(\varepsilon_n z_n)
=V_0.
\]

\section*{Acknowledgments}
%We would like to thank the anonymous referee for his/her careful readings of our manuscript and the useful comments. 
Z. Yang is supported by National Natural Science Foundation of China (12301145, 12261107, 12561020) and Yunnan Fundamental Research Projects (202301AU070144, 202401AU070123).

\medskip
{\bf Data availability:}  Data sharing is not applicable to this article as no new data were created or analyzed in this study.

\medskip
{\bf Conflict of Interests:} The Author declares that there is no conflict of interest.

\bibliographystyle{plain} 
\bibliography{ref} 

\end{document}